\documentclass[11pt,final]{amsart}
\usepackage[active]{srcltx}
\usepackage[latin1]{inputenc}
\usepackage{amsmath}
\usepackage{amsfonts}
\usepackage{amssymb}
\usepackage[colorlinks=true,urlcolor=blue,
citecolor=red,linkcolor=blue,linktocpage,pdfpagelabels,bookmarksnumbered,bookmarksopen]{hyperref}
\usepackage{graphicx, tikz}

\tikzset { domaine/.style 2 args={domain=#1:#2} }
\tikzset{
xmin/.store in=\xmin, xmin/.default=-3, xmin=-3,
xmax/.store in=\xmax, xmax/.default=3, xmax=3,
ymin/.store in=\ymin, ymin/.default=-3, ymin=-3,
ymax/.store in=\ymax, ymax/.default=3, ymax=3,
}

\newcommand {\axes} {
\draw[->] (\xmin,0) -- (\xmax,0);
\draw[->] (0,\ymin) -- (0,\ymax);
}
\newcommand {\fenetre}
{\clip (\xmin,\ymin) rectangle (\xmax,\ymax);}

\usepackage[english]{babel}
\usepackage{marginnote}
\usepackage[left=2.95cm,right=2.95cm,top=2.8cm,bottom=2.8cm]{geometry}
\numberwithin{equation}{section}
\usepackage{mathrsfs}
\usepackage{color}
\usepackage{amsmath, bbm}
\usepackage{amsfonts}
\usepackage{amssymb}
\usepackage{graphicx}%
\providecommand{\U}[1]{\protect\rule{.1in}{.1in}}
\definecolor{linkcolor}{rgb}{0.00,0.50,0.00}
\providecommand{\U}[1]{\protect\rule{.1in}{.1in}}
\newtheorem{theorem}{Theorem}[section]
\newtheorem{proposition}[theorem]{Proposition}
\newtheorem{lemma}[theorem]{Lemma}
\newtheorem{corollary}[theorem]{Corollary}
\newtheorem{definition}{Definition}[section]
\newtheorem{remark}{Remark}[section]

\numberwithin{equation}{section}

\newcommand{\N}{{\mathbb N}}
\newcommand{\R}{{\mathbb R}}

\newcommand{\pical}{\mathcal{P}}

\newcommand{\deb}{\rightharpoonup}

\newcommand{\dve}{\mathrm{div}}

\newcommand{\Om}{\Omega}

\newcommand{\ve}{\varepsilon}
\newcommand{\ro}{\varrho}

\DeclareMathOperator{\argmin}{argmin}
\DeclareMathOperator{\spt}{spt}

\newcommand{\ds}{\displaystyle}

\newcommand{\cC}{{\mathcal C}}
\newcommand{\cF}{{\mathcal F}}
\newcommand{\cL}{{\mathcal L}}
\newcommand{\cP}{{\mathcal P}}

\def\a{\alpha}

\def\g{\gamma}

\def\t{\tau}
\def\HH{\mathcal{H}}
\def\e{\varepsilon}
\newcommand{\eps}{{\varepsilon}}

\def\spt{{\rm{spt}}} 
\def\Leb{{\rm{Leb}}} 
\def\dd{{\rm d}} 
\def\dist{{\rm dist}}

\newcommand{\mres}{\mathbin{\vrule height 1.6ex depth 0pt width
0.13ex\vrule height 0.13ex depth 0pt width 1.3ex}}

\def\tm{\ro} 
\def\gm{g} 

\def\a{\alpha}

\def\g{\gamma}

\def\t{\tau}
\def\HH{\mathcal{H}}
\def\e{\varepsilon}

\def\ind{\mathbbm{1}}

\usepackage{babel}
\reversemarginpar
\begin{document}
\title[BV estimates in optimal transport]{BV estimates in optimal transportation and applications}
\author[G. De Philippis]{Guido De Philippis}
\address{UMPA, CNRS and \'Ecole Normale Sup\'erieure de Lyon, 46, all\'ee d'Italie 
69364 Lyon Cedex 07, France}
\email{guido.de-philippis@ens-lyon.fr}
\author[A. R. M\'esz\'aros]{Alp\'ar Rich\'ard M\'esz\'aros}
\address{Laboratoire de Math\'ematiques d'Orsay, Universit\'e Paris-Sud, 91405 Orsay Cedex, France}
\email{alpar.meszaros@math.u-psud.fr}
\author[F. Santambrogio]{Filippo Santambrogio}
\address{Laboratoire de Math\'ematiques d'Orsay, Universit\'e Paris-Sud, 91405 Orsay Cedex, France}
\email{filippo.santambrogio@math.u-psud.fr}
\author[B. Velichkov]{Bozhidar Velichkov}
\address{Laboratoire LJK, Universit\'e J. Fourier, 51, rue des Math\'ematiques, 
38041 Grenoble Cedex 9, France}\email{bozhidar.velichkov@imag.fr }
\maketitle

\begin{abstract}
In this paper we study the  \(BV\) regularity for solutions of certain variational problems in Optimal Transportation. We prove that the Wasserstein projection of a measure with \(BV\) density on the set of measures with density bounded by a given \(BV\) function $f$ is of bounded variation as well and we also provide a precise estimate of its \(BV\) norm. Of particular interest is the case $f=1$, corresponding to a projection onto a set of densities with an $L^\infty$ bound, where we prove that the total variation decreases by projection. This estimate and, in particular, its iterations have a natural application to some evolutionary PDEs as, for example, the ones describing a crowd motion. In fact, as an application of our results, we obtain \(BV\) estimates for solutions of some non-linear parabolic PDE by means of optimal transportation techniques.We also establish some properties of the Wasserstein projection which are interesting in their own, and allow for instance to prove uniqueness of such a projection in a very general framework.   
\end{abstract}
\section{Introduction}

Among variational problems involving optimal transportation and Wasserstein distances, a very recurrent one is the following
\begin{equation}\label{main pb F}
\min_{\varrho\in \mathcal{P}_2(\Omega)} \frac1 2 W^{2}_2(\varrho,g)+\tau F(\varrho)\,,
\end{equation}
where $F$ is a given functional on probability measures, $\tau>0$ a parameter which can possibly be small, and $g$ is a given probability in $ \mathcal{P}_2(\Omega)$ (the space of probability measures on $\Om\subseteq\R^d$ with finite second moment $\int |x|^2\,\dd\ro(x)<+\infty$). This 
very instance of the problem is exactly the one we face in the time-discretization of the gradient flow of $F$ in $ \mathcal{P}_2(\Omega)$, where $g=\ro^\tau_k$ is the measure at step $k$, and the optimal $\ro$ will be the next measure $\ro^\tau_{k+1}$. Under suitable assumptions, at the limit when $\tau\to 0$, this sequence converges to a curve of measures which is the gradient flow of $F$ (see \cite{AmbGigSav,MovMin} for a general description of this theory). 

The same problem also appears in other frameworks as well, for fixed $\tau$. For instance in image processing, if $F$ is a smoothing functional, this is a model to find a better (smoother) image $\ro$ which is not so far from the original $g$ (the choice of the distance $W_2$ in this case can be justified by robustness arguments), see \cite{carola}. In some urban planning models (see \cite{buttazzo,T+GV}) $g$ represents the distribution of some resources and $\ro$ that of population, which from one side is attracted by the resources $g$ and on the other avoids creating zones of high density thus guaranteeing enough space for each individual. In this case the functional $F$ favors diffused measures, for instance $F(\ro)=\int h(\ro(x))\,\dd x$, where $h$ is a convex and superlinear function, which gives a higher cost to high densities of $\ro$.
Alternatively, $g$ could represent the distribution of population, and $\ro$ that of services, to be chosen so that they are close enough to $g$ but more concentrated. This effect can be obtained by choosing $F$ that favors concentrated measures. 

When $F$ takes only the values $0$ and $+\infty$, \eqref{main pb F} becomes a projection problem. Recently, the projection onto the set\footnote{Here and in the sequel we denote by $K_f$ the set of absolutely continuous measure with density bounded by \(f\): \[K_f:=\{\ro\in \mathcal{P}(\Omega)\,:\,\ro\leq f\dd x \}\]} $K_1$ of densities bounded above by the constant $1$ has received lot of attention. This is mainly due to its applications in the time-discretization of evolution problems with density constraints typically associated to crowd motion. For a precise description of the associated model we refer to \cite{aude phd, MauRouSan}, where a crowd is described as a population of particles which cannot overlap, and cannot go beyond a certain threshold density.

In this paper we concentrate on the case where $F(\ro)=\int h(\ro)$ for a convex integrand $h:\R_+\to\R\cup\{+\infty\}$.
The case of the projection on $K_1$ is obtained by taking 
the following function:
\[
h(\varrho)=
\begin{cases}
0,\quad&\text{if \(0\le \varrho\le 1\)}\\
+\infty, &\text{if \(\varrho> 1\)}\,,
\end{cases}
\]

We are interested in the estimates on the minimizer $\bar\ro$ of \eqref{main pb F}. In general then can be divided into two categories: the ones which are independent of $g$ (but depend on $\tau$) and the ones uniform in $\tau$ (dependent on $g$). A typical example of the first type of estimate can be obtained by writing down the optimality conditions for \eqref{main pb F}. In the case $F(\ro)=\int h(\ro)$, we get $\varphi+\tau h'(\bar\ro)=const$, where $\varphi$ is the Kantorovich potential in the transport from $\bar\ro$ to $g$ (in fact this equality holds only $\bar\ro-$a.e., but we skip the details and just recall the heuristic argument). On a bounded domain, $\varphi$ is Lipschitz continuous with a universal Lipschitz constant depending only on the domain, and so is $\tau h'(\bar\ro)$. If $h$ is strictly convex and $C^1$, then we can deduce the Lipschitz continuity for $\bar\ro$. The bounds on the Lipschitz constant of $\bar\ro$ do not really depend on $g$, but on the other hand they clearly degenerate as $\tau\to 0$.  Another bound that one can prove is $\|\bar\ro\|_{L^\infty}\leq \|g\|_{L^\infty}$ (see \cite{CarSan, T+GV}), which, on the contrary, is independent of $\tau$. 

In this paper we are mainly concerned with \(BV\) estimates. As we expect uniform bounds, in what follows we get rid of the parameter $\tau$.

We recall that for every function \(\varrho\in L^1\) and every open set  \(A\)  the total variation of \(\nabla \varrho\) in \(A\) is defined as 
\[
TV(\varrho,A)=\int_{A} |\nabla \varrho|=\sup\left\{\int \varrho\, \dve\xi\,\dd x\quad : \quad \xi\in C_{c}^1(A),\quad |\xi|\le 1\right\}.
\]
Our main theorem reads as follows:
\begin{theorem}\label{mainthm1}
Let \(\Omega\subset\R^d\) be a  (possibly unbounded) convex set, \(h:\R_+ \to \R\cup\{+\infty\}\) be a convex and l.s.c. function and  \(g\in\mathcal{P}_2(\Omega)\cap BV(\Omega)\). If \(\bar\ro\) is a minimizer of the following variational problem
\begin{equation*}
\min_{\varrho\in \mathcal{P}_2(\Omega)} \frac1 2 W^{2}_2(\varrho,g)+\int_{\Omega} h(\varrho(x))\,\dd x\,,
\end{equation*}
then
\begin{equation}\label{bv1}
\int_{\Omega} |\nabla \bar\ro|\,\dd x\le \int_{\Omega} |\nabla g|\,\dd x\,.
\end{equation}

\end{theorem}
As we said, this covers the case of the Wassertstein projection of \(g\)  on the subset $K_1$ of \(\mathcal P_{2}(\Omega)\) given by the measures  with density less than or equal to \(1\). Starting from Theorem \ref{mainthm1} and constructing an appropriate approximating sequence of functionals we are actually able to establish \(BV\) bounds for more general Wasserstein projections related to a prescribed \(BV\) function \(f\). More  precisely we have the 
following result.

\begin{theorem}\label{mainthm2}
Let \(\Omega\subset\R^d\) be a  (possibly unbounded) convex set,   \(g\in\mathcal{P}_2(\Omega)\cap BV(\Omega)\) and let 
\(f\in BV_{\rm loc} (\Omega)\) be a function with
\[
\int_{\Omega} f\,\dd x\ge 1.
\]
If 
\begin{equation}\label{proj}
\bar\ro={\rm argmin} \Big\{ W_2^2(\varrho,g)\ :\ \varrho\in \mathcal P_2(\Omega),\ \varrho \le f \text{  a.e.} \Big\}\,,
\end{equation}
then 
\begin{equation}\label{milano}
\int_{\Omega} |\nabla \bar\ro|\,\dd x\le \int_{\Omega} |\nabla g|\,\dd x+2 \int_{ \Omega} |\nabla f|\,\dd x.
\end{equation}
\end{theorem}

We would like to spend some words on the $BV$ estimate for the projection on the set $K_1$, which is the original motivation for this paper. We note that this corresponds to the case 
$$h(\ro)=\begin{cases}0,\ \ \quad\text{if}\ \ro\in[0,1],\\ +\infty,\ \text{if}\ \ro>1,\end{cases}$$
in Theorem \ref{mainthm1} and to the case $f=1$ in Theorem \ref{mainthm2}. In both cases we obtain that \eqref{bv1} holds.\smallskip

\begin{minipage}{8.4cm}
\begin{tikzpicture}[scale=0.9,xmin=-4.5,xmax=4.5,ymin=-1,ymax=5]
\axes \fenetre
\draw[domaine={-2.2}{2.5}, samples=62] plot (\x,{1.8*exp(-0.5*abs(\x))*(1.6+cos(100*(\x/(1+\x)))+0.2*sin(400*\x) )-0.5*exp(-0.5*(\x-1)^2)});
\draw[red,domaine={-4.2}{-2.2}, samples=40] plot (\x,{1.8*exp(-0.5*abs(\x))*(1.6+cos(100*(\x/(1+\x)))+0.2*sin(400*\x) )-0.5*exp(-0.5*(\x-1)^2)});
\draw[red,domaine={2.5}{4.2},  samples=40]  plot (\x,{1.8*exp(-0.5*abs(\x))*(1.6+cos(100*(\x/(1+\x)))+0.2*sin(400*\x) )-0.5*exp(-0.5*(\x-1)^2)});\draw[red,domaine={-2.2}{2.5}, samples=100] plot (\x,2);
\draw[red,dotted] (-2.2,2)--(-2.2,0.35);
\draw[red,dotted] (2.5,2)--(2.5,0.7);
\draw [above](0.2,2) node{$1$};
\draw [above](0.7,3) node{$g$};
\draw [above](2.15,2) node{$\bar\ro=1$};
\draw [above](3.75,0.7) node{$\bar\ro=g$};

\fill[color=gray!20, opacity=0.4]
(-4.2,0) -- (-4.2,0.5) 
-- plot[domaine={-4.2}{-2.2}]  (\x,{1.8*exp(-0.5*abs(\x))*(1.6+cos(100*(\x/(1+\x)))+0.2*sin(400*\x) )-0.5*exp(-0.5*(\x-1)^2)})
-- (-2.2,2) -- (2.5,2) -- plot[domaine={2.5}{4.2}]  (\x,{1.8*exp(-0.5*abs(\x))*(1.6+cos(100*(\x/(1+\x)))+0.2*sin(400*\x) )-0.5*exp(-0.5*(\x-1)^2)}) -- (4.2,0) -- cycle; 
\end{tikzpicture}
\end{minipage}
\begin{minipage}{7cm}
\smallskip

In dimension one the estimate \eqref{bv1} can be obtained by some direct considerations. In fact, by \cite{figa2} we have that the constraint $\bar\ro\le 1$ is saturated, i.e. the projection is of the form 
$$\bar\rho(x)=\begin{cases}1,\ \ \quad\text{if}\ x\in A,\\ g(x),\ \text{if}\ x\notin A,\end{cases}$$
for an open set $A\subset\R$. Since we are in dimension one, $A$ is a union of intervals and so it is sufficient to show that \eqref{bv1} holds in the case that $A$ is just one interval,  as in the picture on the left. In this case it is immediate to check that the total variation of $g$ has not increased after the projection since $\bar\ro=1$ on $A$, while there is necessarily a point $x_0\in A$ such that $g(x_0)\ge 1$.
\end{minipage}\medskip

In dimension $d\ge 2$ the estimate \eqref{bv1} is more involved essentially due to the fact that the projection tends to spread in all directions. This geometric phenomenon can be illustrated with the following simple example. Consider the function $g=(1+\ve)\ind_{B(0,R)}$, where $\ve>0$ and $R>0$ are such that $(1+\ve)|B(0,R)|=1$. By the saturation of the constraint and symmetry considerations the projection $\bar\ro$ of $g$ is the characteristic function $\bar\ro=\ind_{B(0,\bar R)}$, where $\bar R=(1+\ve)^{1/d}R$. The total variation involves two opposite effects: the perimeter of the ball increases, but the height of the jump passes from $1+\ve$ to $1$. In fact we have 
$$\int_{\R^d} |\nabla \bar\ro|=d\omega_d \bar R^{d-1}=d\omega_dR^{d-1} (1+\ve)^{(d-1)/d}\le d\omega_d R^{d-1} (1+\ve)=\int_{\R^d} |\nabla g|.$$
Further explicit examples are difficult to construct. Even in the case $g=(1+\ve)\ind_\Omega$, where $\Omega$ is a union of balls, it is not trivial to compute the \(BV\) norm of the projection, which is the characteristic function of a union of (overlapping) balls. 

\medskip

The BV estimates are useful when the projection is treated as one time-step of a discretized evolution process. For instance, a \(BV\) bound allows to transform weak convergence in the sense of measures into strong $L^1$ convergence (see Section 6.3). Also, if we consider a PDE mixing a smooth evolution, such as the Fokker-Planck evolution, and some projection steps (in order to impose a density constraint, as in crowd motion issues), one could wonder which bounds on the regularity of the solution are preserved in time. From the fact that the discontinuities in the projected measure destroy any kind of $W^{1,p}$ norm, it is natural to look for \(BV\) bounds.  Notice by the way that, for these kind of applications, proving $\int_{\Omega} |\nabla \bar\ro|\le \int_{\Omega} |\nabla g|$ (with no multiplicative coefficient nor additional term) is crucial in order to iterate this estimate at every  step.

\medskip

The paper is structured as follows: In Section \ref{2} we recall some preliminary results in optimal transportation, in Section \ref{3} we establish our \emph{main inequality}, in Section \ref{4}  we prove Theorem \ref{mainthm1} while in Section \ref{5} we collect  some properties of solution of \eqref{proj} which can be interesting in their own and we we prove   Theorem \ref{mainthm2}.  Eventually, in Section \ref{6} we present some applications of the above results, connections with other variational and evolution problems and some open questions. 
\medskip

\noindent {\bf Acknowledgments} The authors would like to thank te referee for a careful reading of the manuscript and for her/his comments. The second and third author gratefully acknowledge the support of the ANR project ANR-12-MONU-0013 ISOTACE.

\section{Notations and preliminaries}\label{2}

In this section we collect some facts about optimal transport that we will need in the sequel, referring  the reader to \cite{villani} for more details. We will denote by \(\mathcal P(\Omega)\) the set of probability measures in \(\Omega\) and by \(\mathcal P_2(\Omega)\) the subset of \(\mathcal P(\Omega)\) given by those with finite second moment (i.e. \(\mu\in \mathcal P_2(\Omega)\) if and only if \(\int|x|^2 \dd \mu<\infty\)). We will also use the spaces $\mathcal M(\Omega)$ of finite measures on $\Omega$ and \(L^1_+(\Omega)\) of non-negative functions in $L^1$. Notice $\{ f\in
L^1_+(\Omega)\,:\, \int f(x)\,\dd x= 1\}=L^1_+(\Omega)\cap \mathcal P(\Omega)$.
In the sequel we will {\em always} identify an absolutely continuous measure with its density (for instance writing \(T_\# f\) for \(T_\# (f \dd x)\) and so on..).

\begin{theorem}\label{struct}
Let $\Omega\subset\R^d$ be a given convex set and let $\ro,g\in L^1_+(\Omega)$ be two probability densities on $\Omega$. Then the following hold:
\begin{enumerate}
\item[(i)] The problem
\begin{equation}\label{kantorovich}
\frac12W_2^2(\ro,g):=
\min\left\{\int_{\Om\times\Om}\frac{1}{2} |x-y|^2\,\dd\gamma\;:\;\gamma\in\Pi(\ro ,g)\right\},
\end{equation}
where
$\Pi(\varrho ,g)$ is the set of  {\em transport plans}, i.e. $\Pi(\ro \dd x,g\dd x):=\{\gamma\in\pical(\Om\times\Om):\,(\pi^x)_{\#}\gamma=\ro ,\,(\pi^y)_{\#}\gamma=g\},$ has a unique solution, which is of the form $\gamma_{\hat T}:=(id,\hat T)_\#\ro$, and $\hat T:\Omega\to\Omega$ is a solution of the problem 
\begin{equation}\label{transomega}
\min_{T_\# \ro=g } \int_{\Omega}\frac{1}{2}|x-T(x)|^2\, \ro(x)\,\dd x\,.
\end{equation}
\item[(ii)] The map \(\hat T: \{\ro>0\}\to \{g>0\}\) is a.e. invertible and its inverse $\hat S:=\hat T^{-1}$ is a solution of the problem 
\begin{equation}\label{transomegaS}
\min_{S_\# g=\ro} \int_{\Omega}\frac{1}{2}|x-S(x)|^2\, g(x)\,\dd x. 
\end{equation}
\item[(iii)]  $W_2(\cdot,\cdot)$ is a distance on the space $\pical_2(\Omega)$ of probabilities over $\Omega$ with finite second moment.
\item[(iv)] We have 
\begin{equation}\label{duality}
\frac 12 W_2^2(\ro,g)=\max\left\{ \int_{\Omega}\varphi(x)\ro(x)\,\dd x+\int_\Omega \psi(y)g(y)\,\dd y\ :\  \varphi(x)+\psi(y)\le\frac12 |x-y|^2,\ \forall x, y\in\Omega\right\}.
\end{equation}
\item[(v)]
The optimal functions $\hat\varphi,\hat\psi$ in \eqref{duality} are continuous, differentiable almost everywhere, Lipschitz if $\Omega$ is bounded, and such that:
\begin{itemize}
\item $\hat T(x)=x-\nabla\hat\varphi(x)\quad \hbox{and}\quad \hat S(x)=x-\nabla\hat\psi(x)\quad\hbox{for a.e.}\quad x\in\Omega;$ in particular, the gradients of the optimal functions are uniquely determined (even in case of non-uniqueness of $\hat\varphi$ and $\hat\psi$) a.e. on $\{\ro>0\}$ and $\{g>0\}$, respectively;
\item the functions
$$x\mapsto \frac{|x|^2}{2}-\hat\varphi(x)\quad\hbox{and}\quad x\mapsto \frac{|x|^2}{2}-\hat\psi(x),$$
are convex in $\Omega$ and hence $\hat\varphi$ and $\hat\psi$ are semi-concave;
\item $\ds\hat\varphi(x)=\min_{y\in\Omega}\left\{\frac12|x-y|^2-\hat\psi(y)\right\}\qquad\hbox{and}\qquad \hat\psi(y)=\min_{x\in\Omega}\left\{\frac12|x-y|^2-\hat\varphi(x)\right\};$
\item if we denote by $\chi^c$ the  $c-$transform of a function $\chi:\Omega\to\R$ defined through $\chi^c(y)=\inf_{x\in\Omega}\frac12 |x-y|^2-\chi(x)$, then the maximal value in \eqref{duality} is also equal to 
\begin{equation}\label{cconc}
\max\left\{ \int_{\Omega}\varphi(x)\ro(x)\,\dd x+\int_\Omega \varphi^c(y)g(y)\,\dd y,\;\varphi\in C^0(\Omega)\right\}
\end{equation}
and the optimal $\varphi$ is the same $\hat\varphi$ as above, and is such that $\hat\varphi=(\hat\varphi^c)^c$ a.e. on $\{\ro>0\}$.
\end{itemize}
\item[(vi)] If $g\in \pical_2(\Omega)$ is given, the functional $W:\mathcal \pical_2(\Omega)\to\R$ defined through
$$W(\ro)=\frac 12 W_2^2(\ro,g)=\max\left\{ \int_{\Omega}\varphi(x)\ro(x)\,\dd x+\int_\Omega \varphi^c(y)g(y)\,\dd y,\;\varphi\in C^0(\Omega)\right\}$$
is convex. Moreover, if $\{g>0\}$ is a connected open set we can choose a particular potential $\hat\varphi$, defined as
$$\hat\varphi(x)=\inf\left\{\frac12|x-y|^2-\hat\psi(y)\,:\,y\in \spt (g)\right\},$$
where $\hat\psi$ is the unique (up to additive constants) optimal function $\hat \psi$ in \eqref{duality} 
(i.e. $\hat\varphi$ is the $c-$transform of $\hat\psi$ computed on $\Omega\times\spt(g)$). With this choice, if $\chi=\tilde\ro-\ro$ is the difference between two probability measures, then we have
$$\lim_{\ve\to 0} \frac{W(\varrho+\ve\chi)-W(\varrho)}{\ve}=\int_\Om \hat\varphi\,\dd\chi.$$ As a consequence, $\hat \varphi$ is the first variation of $W$. 
\end{enumerate} 
\end{theorem}

The only non-standard point is the last one (the computation of the first variation of $W$): it is sketched in \cite{buttazzo}, and a more detailed presentation will be part of \cite{OTAM} (Section 7.2). Uniqueness of $\hat \psi$ on $\spt(g)$ is obtained from the uniqueness of its gradient and the connectedness of  $\{g>0\}$. \smallskip

We also need some regularity results on optimal transport maps, see \cite{caffarelli,caffarelli2}.

\begin{theorem}\label{regul}
Let $\Omega\subset\R^d$ be a bounded uniformly convex set with smooth boundary and let $\ro,g\in L^1_+(\Omega)$ be two probability densities on $\Omega$ away from zero and infinity\footnote{We say that $\ro$ and $g$ are away from zero and infinity if there is some $\eps>0$ such that $\eps\le \ro\le 1/\eps$ and $\eps\le g\le 1/\eps$ a.e. in $\Omega$.}. Then, using the notations from Theorem \ref{struct}, we have:
\begin{enumerate}
\item[(i)] $\hat T\in C^{0,\alpha}(\overline \Omega)$ and $\hat S\in C^{0,\alpha}(\overline\Omega)$.
\item[(ii)] If $\ro\in C^{k,\beta}(\overline\Omega)$ and $g\in C^{k,\beta}(\overline\Omega)$, then $\hat T\in C^{k+1,\beta}(\overline\Omega)$ and $\hat S\in C^{k+1,\beta}(\overline\Omega)$.
\end{enumerate}
\end{theorem}

Most of our proofs will be done by approximation. To do this, we need a stability result

\begin{theorem}\label{approx}
Let $\Omega\subset\R^d$ be a bounded convex set and let $\ro_n\in L^1_+(\Omega)$ and  $g_n\in L^1_+(\Omega)$ be two sequences of probability densities in $\Omega$. Then, using the notations from Theorem \ref{struct}, if $\ro_n\rightharpoonup \ro$ and $g_n\rightharpoonup g$ weakly as measures, then we have:
\begin{enumerate}
\item[(i)] $W_2(\ro,g)=\lim_{n\to\infty}W_2(\ro_n,g_n).$
\item[(ii)] there exist two semi-concave functions $\varphi,\psi$ such that $\nabla \hat\varphi_n\to\nabla\varphi\qquad\hbox{and}\qquad \nabla \hat\psi_n\to \nabla \psi$ a.e. and $\nabla\varphi=\nabla\hat\varphi$ a.e. on $\{\ro>0\}$ and  $\nabla\psi=\nabla\hat\psi$ a.e. on $\{g>0\}$.

\end{enumerate} 
If $\Omega$ is unbounded (for instance $\Omega=\R^d$), then the convergence $\ro_n\rightharpoonup \ro$ and $g_n\rightharpoonup g$ weakly as measures is not enough to guarantee  {\rm (i)} but only implies $W_2(\ro,g)\leq\liminf_{n\to\infty}W_2(\ro_n,g_n).$ Yet, {\rm (i)} is satisfied if $W_2(\ro_n,\ro),W_2(g_n,g)\to 0$, which is a stronger condition. 
\end{theorem}

\begin{proof}
The proof of {\it (i)} can be found in \cite{villani}. We prove {\it (ii)}. (Actually this is a consequence of the Theorem 3.3.3. from \cite{cannarsa}, but for the sake of completeness we sketch its simple  proof).

We first note that due to Theorem \ref{struct} {\it (v)} the sequences $\hat\varphi_n$ and $\hat\psi_n$ are equi-continuous. Moreover, since the Kantorovich potentials are uniquely determined up to a constant we may suppose that there is $x_0\in\Omega$ such that $\hat\varphi_n(x_0)=\hat\psi_n(x_0)=0$ for every $n\in\N$. Thus, $\hat\varphi_n$ and $\hat\psi_n$ are locally uniformly bounded in $\Omega$ and, by the Ascoli-Arzel\`a Theorem, they converge uniformly up to a subsequence 
$$\hat\varphi_n\xrightarrow[n\to\infty]{} \varphi_\infty\qquad\hbox{and}\qquad\hat\psi_n\xrightarrow[n\to\infty]{} \psi_\infty,$$
to some continuous functions $\varphi_\infty,\psi_\infty\in C(\Omega)$, satisfying 
$$\varphi_\infty(x)+\psi_\infty(y)\le\frac12 |x-y|^2,\quad\hbox{for every}\quad x,y\in\Omega.$$

In order to show that $\varphi_\infty$ and $\psi_\infty$ are precisely Kantorovich potentials, we use the characterization of the potentials as solutions to the problem \eqref{duality}. Indeed, let $\varphi$ and $\psi$ be such that $\varphi(x)+\psi(y)\le\frac12 |x-y|^2$ for every $ x,y\in\Omega$. Then, for every $n\in\N$ we have 
\begin{equation*}
\int_{\Omega}\hat\varphi_n(x)\ro_n(x)\,\dd x+\int_\Omega \hat\psi_n(y)g_n(y)\,\dd y\ge \int_{\Omega}\varphi(x)\ro_n(x)\,\dd x+\int_\Omega \psi(y)g_n(y)\,\dd y, 
\end{equation*}
and passing to the limit we obtain
\begin{equation*}
\int_{\Omega}\varphi_\infty(x)\ro(x)\,\dd x+\int_\Omega \psi_\infty(y)g(y)\,\dd y\ge \int_{\Omega}\varphi(x)\ro(x)\,\dd x+\int_\Omega \psi(y)g(y)\,\dd y, 
\end{equation*}
which proves that $\varphi_\infty$ and $\psi_\infty$ are optimal. In particular, the gradient of these functions coincide with those of $\hat\varphi$ and $\hat\psi$ on the sets where the densities are strictly positive.

We now prove that $\nabla\hat\varphi_n\to \nabla\varphi_\infty$ a.e. in $\Omega$. We denote with $\mathcal{N}\subset\Omega$ the set of points $x\in\Omega$, such that there is a function among $\hat\varphi$ and $\hat\varphi_n$, for $n\in\N$, which is not differentiable at $x$. We note that by Theorem \ref{struct} {\it (v)} the set $\mathcal{N}$ has Lebesgue measure zero. Let now $x_0\in \Omega\setminus \mathcal{N}$ and suppose, without loss of generality, $x_0=0$. Setting 
$$\alpha_n(x):=\frac{|x|^2}{2}-\hat\varphi_n(x)+\hat\varphi_n(0)+x\cdot \nabla\varphi_\infty(0)\quad\hbox{and}\quad \alpha(x):=\frac{|x|^2}{2}-\varphi_\infty(x)+\varphi_\infty(0)+x\cdot \nabla\varphi_\infty(0),$$ 
we have that $\alpha_n$ are all convex and such that $\alpha_n(0)=0$, and hence $\alpha_n(x)\geq \nabla\alpha_n(0)\cdot x$. Moreover, $\alpha_n\to\alpha$ locally uniformly and $\nabla\alpha(0)=0$. Suppose by contradiction that $\lim_{n\to\infty}\nabla \alpha_n(0)\neq 0$. Then, there is a unit vector $p\in\R^d$ and a constant $\delta>0$ such that, up to a subsequence, $p\cdot\nabla \alpha_n\ge \delta$ for every $n>0$. Then, for every $t>0$ we have 
$$\frac{\alpha(pt)}{t}=\lim_{n\to\infty}\frac{\alpha_n(pt)}t\ge\liminf_{n\to\infty} \big\{p\cdot\nabla\alpha_n(0)\big\}\ge\delta,$$
which is a contradiction with the fact that $\nabla\alpha(0)=0$.
\end{proof}

In order to handle our approximation procedures, we also need to spend some words on the notion of $\Gamma-convergence$ (see \cite{introgammaconve}).

\begin{definition}
On a metric space $X$ let $F_n:X\to\R\cup\{+\infty\}$ be a sequence of functions. We define the two lower-semicontinuous functions $F^-$ and $F^+$ (called $\Gamma-\liminf$ and $\Gamma-\limsup$ of this sequence, respectively) by
\begin{gather*}
F^-(x):=\inf\{\liminf_{n\to\infty} F_n(x_n)\;:\;x_n\to x\},\\
 F^+(x):=\inf\{\limsup_{n\to\infty} F_n(x_n)\;:\;x_n\to x\}.
 \end{gather*}
Should $F^-$ and $F^+$ coincide, then we say that $F_n$ actually $\Gamma-$converges to the common value $F=F^-=F^+$. 
\end{definition}
This means that, when one wants to prove $\Gamma-$convergence of $F_n$ towards a given functional $F$, one has actually to prove two distinct facts: first we need $F^-\geq F$ (this is called $\Gamma-$liminf inequality, i.e. we need to prove $\liminf_n F_n(x_n)\geq F(x)$ for any approximating sequence $x_n\to x$) 
and then $F^+\leq F$ (this is called $\Gamma-$limsup inequality, i.e. we need to find a {\em recovery sequence} $x_n\to x$ such that $\limsup_n F_n(x_n)\leq F(x)$).  

The definition of $\Gamma-$convergence for a continuous parameter $\ve\to 0$ obviously passes through the convergence to the same limit for any subsequence $\ve_n\to 0$.

Among the properties of $\Gamma-$convergence we have the following:
\begin{itemize}
\item if there exists a compact set $K\subset X$ such that $\inf_X F_n=\inf_K F_n$ for any $n$, then $F$ attains its infimum and $\inf F_n\to \min F$,
\item if $(x_n)_n$ is a sequence of minimizers for  $F_n$ admitting a subsequence converging to $x$, then $x$ minimizes $F$ (in particular, if $F$ has a unique minimizer $x$ and the sequence of minimizers $x_n$ is compact, then $x_n\to x$),
\item if $F_n$ is a sequence $\Gamma-$converging to $F$, then $F_n+G$ will $\Gamma-$converge to $F+G$ for any continuous function $G:X\to\R\cup\{+\infty\}$.
\end{itemize}

In the sequel we will need the following  two easy criteria to guarantee $\Gamma-$convergence.
\begin{proposition}
If each $F_n$ is l.s.c. and $F_n\to F$ uniformly, then $F_n$  $\Gamma-$converges to $F$.

If each $F_n$ is l.s.c., $F_n\leq F_{n+1}$ and $F(x)=\lim_n F_n(x)$ for all $x$,  then $F_n$  $\Gamma-$converges to $F$.
\end{proposition}
We will essentially  apply the notion of $\Gamma-$convergence in the space $X=\pical(\Omega)$ endowed with the weak convergence\footnote{We say that a family of probability measure \(\mu_n\) weakly converges to a probability measure \(\mu\) in \(\Omega\)  if \(\int \varphi\,\dd\mu_n\to \int\varphi\,\dd \mu\) for all $ \varphi\in C_b(\Omega)$, where \(C_b(\Omega)\) is the space of continuous and bounded functions on \(\Omega\).} (which is indeed metrizable on this bounded subset of the Banach space of measures) since if, instead, we endowed the space $\pical_2(\Om)$ with the $W_2$ convergence, then we would lack compactness whenever $\Om$ is not compact itself.
\bigskip

We conclude this section with the following simple lemma concerning properties of the functional 
\[
\mathcal M(\Omega)\ni \varrho \mapsto H(\varrho)=\begin{cases}\int_\Omega h(\varrho(x))\, \dd x,&\mbox{ if }\varrho\ll\dd x,\\
													+\infty,&\mbox{ otherwise.}\end{cases}
\]

\begin{lemma}\label{lemma:h}Let \(\Omega\) be an open set and  \(h:\R\to \mathbb R\cup \{+\infty\}\) be convex, l.s.c. and superlinear at $+\infty$, then the functional \(H: \mathcal M(\Omega)\to  \mathbb R\cup \{+\infty\}\)  is convex and lower semicontinuous with respect to the weak convergence of measures. Moreover if \(h\in C^1\) then we have
$$\lim_{\ve\to 0} \frac{H(\varrho+\ve\chi)-H(\varrho)}{\ve}=\int h'(\varrho)\,\dd\chi$$
whenever $\rho,\chi\ll \dd x$,  $H(\varrho)<+\infty$ and $H(\varrho+\ve\chi)<+\infty$ at least for small $\ve$. As a consequence, $h'(\varrho)$ is the first variation of $H$. 
\end{lemma}
For this classical fact, and in particular for the semicontinuity, we refer to \cite{But book} and \cite{boubu}.

We also use this lemma, together with point (vi) in Theorem \ref{struct} to deduce the following optimality conditions.
\begin{corollary}\label{coro:h and W}
Let \(\Omega\) be a bounded open set, $g\in L^1_+(\Om)$ an absolutely continuous and strictly positive probability density on $\Om$, the potential $\hat\varphi$ and the functional $W$ defined as in point (vi) in Theorem \ref{struct}. Let \(h:\R\to \mathbb R\) be a $C^1$ convex and superlinear function, and let \(H: \mathcal M(\Omega)\to  \mathbb R\cup \{+\infty\}\) be defined as above. Suppose that $\bar\ro$ solves the minimization problem
$$\min\{W(\ro)+H(\ro)\;:\;\ro\in\pical(\Om)\}.$$
Then there exists a constant $C$ such that 
$$h'(\bar\ro)=\max\{(C-\hat\varphi),h'(0)\}.$$\end{corollary}
The proof of this fact is contained in \cite{buttazzo} and in Section 7.2.3 of \cite{OTAM}. We give a sketch here.
\begin{proof}
Take an arbitrary competitor $\tilde\ro$, define $\ro_\eps:=(1-\eps)\bar\ro+\eps\tilde\ro$ and $\chi=\tilde\ro-\bar\ro$ and write the optimality condition
$$0\leq \lim_{\ve\to 0} \frac{(H+W)(\bar\varrho+\ve\chi)-(H+W)(\bar\varrho)}{\ve}.$$
This implies
$$\int (\hat\varphi+h'(\bar\ro))\,\dd\tilde\ro\geq \int (\hat\varphi+h'(\bar\ro))\,\dd\bar\ro$$
for any arbitrary competitor $\tilde\ro$. This means that there is a constant $C$ such that $\hat\varphi+h'(\bar\ro)\geq C$ with $\hat\varphi+h'(\bar\ro)=C$ on $\{\bar\ro>0\}$. The claim is just a re-writing of this fact, distinguishing the set where $\bar\ro> 0$ (and hence $h'(\bar\ro)\geq h'(0)$) and the set where $\bar\ro=0$.
\end{proof}
%

\section{The main inequality}\label{3}
In this section we establish the key inequality needed in the proof of Theorems \ref{mainthm1} and \ref{mainthm2}. 
%

\begin{lemma}\label{mainest}
Suppose that $\ro,g\in L^1_+$ are smooth probability densities, which are bounded away from $0$ and infinity, $\Om\subset\R^d$ a bounded and uniformly convex domain and let $H\in C^2(\Omega)$ be a convex function. Then we have the following inequality
\begin{equation}\label{maineste1}
\int_{\Omega}\Big(\ro\, \nabla\cdot\big[\nabla H(\nabla \varphi)\big]-g\,  \nabla\cdot\big[ \nabla H(-\nabla \psi)\big]\Big)\,\dd x\le 0,
\end{equation}
where $(\varphi,\psi)$ is a choice of Kantorovich potentials.
\end{lemma}
\begin{proof}
We first note that since $\ro$ and $g$ are smooth  and away from zero and infinity in $\Omega$, Theorem \ref{regul} implies that $\varphi,\psi $ are smooth as well. 
Now using the identity $S(T(x))\equiv x$ and that \(S_\# g=\ro\) we get 
\begin{align*}
\int_{\Omega}\ro(x)\,\nabla \cdot\big[\nabla H(\nabla \varphi(x))\big]\,\dd x&=\int_{\Omega}g(x)\, \Big[\nabla \cdot\big[\nabla H(\nabla \varphi)\big]\Big](S(x))\,\dd x\\
& = \int_{\Omega} g(x)\, \nabla \cdot\Big[\nabla H\big(\nabla \varphi\circ S\big)\Big](x)\,\dd x\\
& \quad+\int_{\Omega}g(x) \left( \Big[\nabla \cdot\big[\nabla H(\nabla \varphi)\big]\Big](S(x))-\nabla \cdot\big[\nabla H\big(\nabla \varphi\circ S\big)\big](x)\right)\,\dd x,
\end{align*}
and, by the equality 
$$-\nabla\psi(x)=S(x)-x=S(x)-T(S(x))=\nabla\varphi(S(x)),$$
we obtain 
\begin{equation}\label{maineste97}
\begin{array}{ll}
\ds\int_{\Omega}\Big(\ro\, \nabla\cdot\big[\nabla H(\nabla \varphi)\big]-g\,  \nabla\cdot\big[ \nabla H(-\nabla \psi)\big]\Big)\,\dd x=\\
\\
\ds\qquad\qquad\qquad\qquad=\int_{\Omega}g(x) \left( \Big[\nabla \cdot\big[\nabla H(\nabla \varphi)\big]\Big](S(x))-\nabla \cdot\big[\nabla H\big(\nabla \varphi\circ S\big)\big](x)\right)\,\dd x\\
\\
\ds\qquad\qquad\qquad\qquad=\int_{\Omega}\ro(x) \left( \nabla \cdot\big[\nabla H(\nabla \varphi)\big]-\Big[\nabla \cdot\big[\nabla H\big(\nabla \varphi\big)\circ S\big]\Big]\circ T \right)\,\dd x.
\end{array}
\end{equation}
For simplicity we set 
\begin{equation}\label{maineste99}
\begin{array}{ll}
\ds E&=\nabla \cdot\left(\nabla H(\nabla \varphi)\right)- \big[\nabla \cdot\left(\nabla H(\nabla \varphi)\circ S\right)\big]\circ T\\
\\
&=\nabla\cdot\xi-\big[\nabla \cdot (\xi\circ S)\big]\circ T,
\end{array}
\end{equation}
where by $\xi$ we denote the continuously differentiable function
$$\ds\xi(x)=(\xi^1,\dots,\xi^d):=\nabla H(\nabla\varphi(x)),$$
whose derivative is given by 
$$\ds D\xi=D\big(\nabla H(\nabla\varphi)\big)=D^2 H(\nabla\varphi)\cdot D^2\varphi.$$ 
We now calculate
\begin{equation}\label{maineste100}
\begin{array}{ll}
\ds\big[\nabla\cdot (\xi\circ S)\big]\circ T&\ds =\sum_{i=1}^d\frac{\partial(\xi^i\circ S)}{\partial x^i}\circ T=\sum_{i=1}^d\sum_{j=1}^d\frac{\partial \xi^i}{\partial x^j}(S(T))\frac{\partial S^j}{\partial x^i}\circ T\\
\\
&\ds ={\rm tr}\left(D\xi\cdot (D T)^{-1}\right)={\rm tr}\left(D^2 H(\nabla\varphi)\cdot D^2\varphi\cdot(I_d-D^2\varphi)^{-1}\right),
\end{array}
\end{equation}
where the last two equality follow by $DS\circ T = (DT)^{-1}$ and we also used that $(DT)^{-1}=(I_d-D^2\varphi)^{-1}$, where $I_d$ is the $d$-dimensional identity matrix.
%

By \eqref{maineste99} and \eqref{maineste100} we have that
\begin{align*}
E&={\rm tr}\big[D^2 H(\nabla\varphi)\cdot D^2\varphi\cdot\left(I_d-(I_d-D^2\varphi)^{-1}\right)\big]\\
& =-{\rm tr}\big[D^2 H(\nabla\varphi)\cdot \big[D^2\varphi\big]^2\cdot (I_d-D^2\varphi)^{-1}\big].
\end{align*}
Since we have that 
$$I_d-D^2\varphi\ge 0,$$
and that the trace of the product of two positive matrices is positive, we obtain $E\le 0$, which together with \eqref{maineste97} concludes the proof.
\end{proof}

\begin{lemma}\label{main}
Let $\Omega\subset\R^d$ be bounded and convex, $\ro,g\in W^{1,1}(\Omega)$ be two  probability densities and $H\in C^2(\R^d)$ be a radially symmetric convex function. 
Then the following inequality holds
\begin{equation}\label{maine1}
\int_{\Omega} \Big(\nabla \ro\cdot\nabla H(\nabla \varphi) +\nabla g\cdot \nabla H(\nabla \psi)\Big)\,\dd x\ge 0,
\end{equation}
where $(\varphi,\psi)$ is a choice of Kantorovich potentials.
\end{lemma}
\begin{proof}
Let us start observing that, due to the radial symmetry of $H$,
\begin{equation}\label{maine2}
\nabla H(\nabla\psi)=-\nabla H(-\nabla\psi).
\end{equation}

\emph{Step 1. Proof in the smooth case.}
Suppose that the probability densities $\ro$ and $g$ are smooth and bounded away from zero and infinity and that $\Omega$ is uniformly convex. As in Lemma \ref{mainest}, we note that under these assumption on $\ro$ and $g$ the Kantorovich potentials are smooth, hence after integration by part  the left hand side of  \eqref{maine1} becomes
\begin{equation*}
\begin{array}{ll}
\ds\int_{\Omega} \Big(\nabla \ro\cdot\nabla H(\nabla \varphi) +\nabla g\cdot \nabla H(\nabla \psi)\Big)\,\dd x &\ds =  \int_{\partial\Omega}\Big(\ro\,\nabla H(\nabla\varphi)\cdot n+g\, \nabla H(\nabla\psi)\cdot n\Big)\,\dd\HH^{d-1}\\
\\
&\ds\qquad -\int_{\Omega} \Big(\ro\,\nabla\cdot \big[\nabla H(\nabla \varphi)\big] + g\,\nabla\cdot \big[\nabla H(\nabla \psi)\big]\Big)\,\dd x \\
\\
&\ds \ge \int_{\partial\Omega}\Big(\ro\,\nabla H(\nabla\varphi)+g\, \nabla H(\nabla\psi)\Big)\cdot n\,\dd\HH^{d-1},
\end{array}
\end{equation*}
where we used Lemma \ref{mainest} and \eqref{maine2}. Moreover, by the radial symmetry of $H$ one has $\nabla H(z)=c(z) z$,  for some $c(z)>0$.  Since the gradients of the Kantorovich potentials $\nabla\varphi$ and $\nabla\psi$ calculated in boundary points are pointing outward $\Omega$ (since $T(x)=x-\nabla\varphi(x)\in\Omega$, and $S(x)=x-\nabla\psi(x)\in\Omega$) we have that 
$$\nabla H(\nabla\varphi(x))\cdot n(x)\ge0\qquad \hbox{and} \qquad \nabla H(\nabla\psi(x))\cdot n(x)\ge0,\qquad \forall x\in\partial\Omega,$$
which concludes the proof of \eqref{maine1} if \(\ro\) and \(g\) are smooth.\\

\emph{Step 2. Withdrawing smoothness and uniform convexity assumptions.}
We first note that for every $\varepsilon>0$ there is a sequence of uniformly convex domains $\Omega_\eps$ such that $\Omega\subset\Omega_\eps\subset \Omega'$ (where $\Omega'$ is a larger fixed convex domain) and $|\Omega_\eps\setminus\Omega|\to 0$, together with smooth nonnegative functions $\ro_\eps\in C^1(\overline\Omega')$ and $g_\eps\in C^1(\overline\Omega')$ such that 
$$\ro_\eps\xrightarrow[\eps\to0]{W^{1,1}(\Omega')}\ro\qquad \hbox{and}\qquad g_\eps\xrightarrow[\eps\to0]{W^{1,1}(\Omega')}g.$$ We will suppose that both $\ro_\eps$ and $g_\eps$ are probability densities on $\Omega_\eps$.
Moreover, by adding a positive constant and then multiplying by another one, we may assume that $\ro_\eps$ and $g_\eps$ are probability densities away from zero:
$$\ro_\eps\ge \eps,\qquad g_\eps\ge \eps\qquad\hbox{and}\qquad\int_{\Omega_\eps}\ro_\eps\,\dd x=\int_{\Omega_\eps} g_\eps\,\dd x=1.$$
Let $\varphi_\eps\in C^{2,\beta}(\overline\Omega_\eps)$ and $\psi_\eps\in C^{2,\beta}(\overline\Omega_\eps)$ be the Kantorovich potentials corresponding to the optimal transport maps between $\ro_\eps$ and $g_\eps$. By \emph{Step 1} we have  
\begin{equation}\label{maine1000}
\int_{\Omega_\eps} \Big(\nabla \ro_\eps\cdot\nabla H(\nabla \varphi_\eps) +\nabla g_\eps\cdot \nabla H(\nabla \psi_\eps)\Big)\,\dd x\ge 0.
\end{equation}
Note that from the boundedness of $\Omega'$ we infer $|\nabla\varphi_\ve|,|\nabla\psi_\ve|\leq C$. Moreover,  $\nabla H$ is locally bounded, which also implies $|\nabla H(\nabla \varphi_\eps)|,|\nabla H(\nabla \psi_\eps)|\leq C$.
On the other hand, from $|\Omega_\eps\setminus\Omega|\to 0$, supposing that the convergence $\nabla\ro_\ve\to \nabla\ro$ and $\nabla g_\ve \to \nabla g$ holds a.e. and is dominated, when we pass to the limit as $\eps\to 0$ the integral restricted to $\Omega_\eps\setminus\Omega$ is negligible. On $\Omega$ we use Theorem \ref{approx}, the bounds on $|\nabla H(\nabla \varphi_\eps)|,|\nabla H(\nabla \psi_\eps)|$ and 
$$\nabla \varphi_\eps\xrightarrow[\eps\to0]{a.e.}\nabla \varphi \qquad\hbox{and}\qquad \nabla \psi_\eps\xrightarrow[\eps\to0]{a.e.}\nabla \psi.$$ 
Passing to the limit as $\eps\to 0$ in \eqref{maine1000} we obtain \eqref{maine1}, which concludes the proof.
%
\end{proof}

\begin{remark}
{\rm
In Lemma  \ref{main} we can drop the  convexity assumption on  $\Om$ if $\ro,g$ have compact support: indeed, it is enough to choose a ball $\Omega'\supset\Omega$ containing the supports of $\ro$ and $g$. 
}
\end{remark}
 \begin{remark}
 {\rm 
Lemma \ref{main}  also remains true in the case of compactly supported densities $g$ and $\ro$, even if we drop the radiality assumption $H(z)=H(|z|)$. In this case  the inequality becomes
$$\int_{\R^d}\Big(\nabla \ro\cdot\nabla H(\nabla \varphi)-\nabla g\cdot \nabla H(-\nabla \psi)\Big)\,\dd x\ge 0.$$
}
\end{remark}
\begin{proof}
The proof follows the same scheme of that of Lemma \ref{main}, first in the smooth case and then for approximation. We select a convex domain $\Om$ large enough to contain the supports of $\ro$ and $g$ in its interior: all the integrations and integration by parts are performed on $\Om$. The only difficulty is that we cannot guarantee the boundary term to be positive. Yet, we first take $\ro,g$ to be smooth and we approximate them by taking $\ro_\ve:=\ve\frac{1}{|\Omega|}+(1-\ve)\ro$ and $g_\ve:=\ve\frac{1}{|\Omega|}+(1-\ve)g$. For these densities and their corresponding potentials $\varphi_\ve,\psi_\ve$, we obtain the inequality
\begin{equation*}
\int_{\Omega} \Big(\nabla \ro_\ve\cdot\nabla H(\nabla \varphi_\ve) +\nabla g_\ve\cdot \nabla H(\nabla \psi_\ve)\Big)\,\dd x   \ge \int_{\partial\Omega}\Big(\ro_\ve\,\nabla H(\nabla\varphi_\ve)+g_\ve\, \nabla H(\nabla\psi_\ve)\Big)\cdot n\,\dd\HH^{d-1}.
\end{equation*}
We can pass to the limit (by dominated convergence as before) in this inequality, and notice that the r.h.s. tends to $0$, since $|\nabla H(\nabla\varphi_\ve)|,|\nabla H(\nabla\psi_\ve)|\leq C$ and $\ro_\ve=g_\ve=\ve/|\Om|$ on $\partial\Omega$. Once the inequality is proven for smooth $\ro,g$, a new approximation gives the desired result.
\end{proof}

We observe that a particular case of Theorem \ref{main}, which we present here as a corollary, could have been obtained in a very different way.

\begin{corollary}\label{coralt}
Let $\Omega\subset\R^d$ be a given bounded convex set and $\ro,g\in W^{1,1}(\Omega)$ be two  probability densities. Then the following inequality holds
\begin{equation}\label{coralt1}
\int_{\Omega} \big(\nabla \ro\cdot \nabla \varphi +\nabla g\cdot \nabla \psi\big)\,\dd x\ge 0,
\end{equation}
where $\varphi$ and $\psi$ are the corresponding Kantorovich potentials.
\end{corollary}
\begin{proof}
The inequality \eqref{coralt1} follows by setting $\ds H(z):=\frac{1}{2}|z|^2$ in Theorem \ref{main}. Nevertheless, in this particular case, there is an alternate proof, using the geodesic convexity of the entropy functional, which we sketch below for $\Omega=\R^d$. 

Consider the entropy functional $\mathcal E:\cP_2(\R^d)\to \R$ defined by 
\begin{equation*}
\mathcal E(\ro)=\begin{cases}\begin{array}{ll}\ds\int_{\R^d}\ro\log\ro\,\dd x,&\ds \hbox{if}\quad\ro\ll\cL^d,\\
\ds+\infty,& \ds\hbox{otherwise}, 
\end{array}
\end{cases}
\end{equation*}
and the geodesic
$$[0,1]\ni t\mapsto \ro_t\in \cP_2(\R^d),\qquad \ro_0=\ro,\qquad \ro_1=g,$$  in the Wasserstein space $(\cP_2(\R^d),W_2)$. It is well known (see, for example, \cite{AmbGigSav}) that the map $t\mapsto \mathcal E(\ro_t)$ is convex and that $\ro_t$ solves the continuity equation 
$$\partial_t\ro_t+\nabla\cdot(\ro_tv_t)=0,\qquad \ro_0=\ro,\qquad \ro_1=g,$$ 
associated to the vector field $v_t=(T-id)\circ ((1-t)id+tT)^{-1}$ induced by the optimal transport map $T=id-\nabla\varphi$ between $\ro$ and $g$. Now since the time derivative of $\mathcal E(\ro_t)$ is increasing, we get
\begin{align*}
-\int_{\R^d}\nabla\ro\cdot\nabla\varphi\,\dd x=\int_{\R^d}\ro v_0\cdot\frac{\nabla\ro}{\ro}\,\dd x&=\frac{d}{dt}\Big|_{\{t=0\}}\mathcal E(\ro_t)\\
&\le \frac{d}{dt}\Big|_{\{t=1\}}\mathcal E(\ro_t)= \int_{\R^d}g v_1\cdot\frac{\nabla g}{g}\,\dd x=\int_{\R^d}\nabla g\cdot\nabla\psi\,\dd x,
\end{align*}
which proves the claim.
\end{proof}

By approximating \(H(z)=|z|\) with \(H(z)=\sqrt{ \eps^2+|z|^2}\), Lemma \ref{main} has the following  useful corollary, where we use the convention   $\ds\frac{z}{|z|}=0$ for $z=0$.

\begin{corollary}\label{cormain}
Let $\Omega\subset\R^d$ be a given bounded convex set and $\ro,g\in W^{1,1}(\Omega)$ be two  probability densities. Then the following inequality holds
\begin{equation}\label{cormain1}
\int_{\Omega} \Big(\nabla \ro\cdot \frac{\nabla \varphi}{|\nabla\varphi|} +\nabla g\cdot \frac{\nabla \psi}{|\nabla \psi|}\Big)\,\dd x\ge 0,
\end{equation}
where $\varphi$ and $\psi$ are the corresponding Kantorovich potentials.
\end{corollary}

\section{BV estimates for minimizers}\label{4}
In this section we prove Theorem \ref{mainthm1}. Since we will need to perform several  approximation arguments, and we want to use $\Gamma-$convergence, we need to provide uniqueness of the minimizers. The following easy lemma is well-known among specialists.

\begin{lemma}\label{strict conv W_2}
Let \(g\in \mathcal P(\Omega)\cap L^1_+(\Omega)\), then 
the functional $\mu \mapsto W_2^2(\mu,g)$ is strictly  convex on $\pical_2(\Om)$.
\end{lemma}
\begin{proof}
Suppose by contradiction that there exist $\mu_0\neq\mu_1$ and $t\in ]0,1[$ are such that 
\[
W_2^2(\mu_t, g )=(1-t)W_2^2(\mu_0, g )+tW_2^2(\mu_1, g ),
\]
 where $\mu_t=(1-t)\mu_0+t\mu_1$. Let $\gamma_0$ be the optimal transport plan in the transport from $\mu_0$ to $g$ (pay attention to the direction: it is a transport map if we see it backward: {\it from $ g $ to $\mu_0$}). As the starting measure is absolutely continuous, by Brenier's Theorem,  $\gamma_0$ is of the form $(T_0,id)_\# g $. Analogously,  take $\gamma_1=(T_1,id)_\# g $ optimal from  $\mu_1$ to $ g $. 
Set $\gamma_t:=(1-t)\gamma_0+t\gamma_1\in \Pi(\mu_t, g )$. We have
\begin{multline*}
(1-t)W_2^2(\mu_0, g )+tW_2^2(\mu_1, g )=W_2^2(\mu_t, g )\leq \int |x-y|^2\,\dd\gamma_t
=(1-t)\int |x-y|^2\,\dd\gamma_0+t\int |x-y|^2\,\dd\gamma_1\\
=(1-t)W_2^2(\mu_0, g )+tW_2^2(\mu_1, g ),
\end{multline*}
which implies that $\gamma_t$ is actually optimal in the transport from $ g $ to $\mu_t$. Yet $\gamma_t$ is not induced from a transport map, unless $T_0=T_1$ a.e. on  \(\{g>0\}\). This is a contradiction with $\mu_0\neq\mu_1$ and proves strict convexity.
\end{proof}

Let us denote by $\mathcal C$ the class of convex l.s.c. function $h:\R_+\to\R\cup\{+\infty\}$,
finite in a neighborhood of $0$ and with finite right derivative $h'(0)$ at $0$, and 
superlinear at $+\infty$.
%
%

\begin{lemma}
If $h\in\mathcal C$ there exists a sequence of $C^2$ convex functions $h_n$, superlinear at $\infty$, with $h_n''>0$, $h_n\leq h_{n+1}$ and $h(x)=\lim_n h_n(x)$ for every $x\in\R_+$.

Moreover, if $h:\R_+\to\R\cup\{+\infty\}$ is a convex l.s.c. superlinear function, there exists a sequence of functions $h_n\in\mathcal C$ with $h_n\leq h_{n+1}$ and $h(x)=\lim_n h_n(x)$ for every $x\in\R_+$.

\end{lemma}

\begin{proof}
Let us start from the case $h\in\mathcal C$.
Set $\ell^+:=\sup\{x\,:\,h(x)<+\infty\}\in\R_+\cup\{+\infty\}$. Let us define an increasing function $\xi_n:\R\to\R$ in the following way:  
$$\xi_n(x):=\begin{cases} h'(0)&\mbox{ for }x\in]-\infty, 0]\\
					h'(x)&\mbox{ for }x\in[0,\ell^+-\frac 1n]\\
					h'(\ell^+-\frac1n)&\mbox{ for }\ell^+-\frac1n\leq x <\ell^+,\\
					h'(\ell^+-\frac1n)+n(x-\ell^+)&\mbox{ for }x\geq \ell^+,\end{cases}$$
where, if the derivative of $h$ does not exist somewhere, we just replace it with the right derivative. (Notice that  when $\ell^+=+\infty$, the last two cases do not apply).

Let $q\geq 0$ be a $C^1$ function with $\spt (q)\subset [-1,0]$, $\int q(t)\,\dd t=1$ and  let us set $q_n(t)=nq(nt)$. We  define $h_n$ as the primitive of the $C^1$ function 
$$h'_n(x):=\int \left(\xi_n(t)-\frac 1n e^{-t}\right)q_n(t-x)\,\dd t,$$
with $h_n(0)=h(0)$.
It is easy to check that all the required properties are satisfied: we have $h_n''(x)\geq \frac 1n e^{-x}$, $h_n$ is superlinear because $\lim_{x\to\infty}\xi_n(x)=+\infty$, and we have increasing convergence $h_n\to h$. 
%
%

For the case of a generic function $h$, it is possible to approximate it with functions in $\mathcal C$ if we define $\ell^-:=\inf\{x\,:\,h(x)<+\infty\}\in\R_+$ and take 
$$h_n(x)=\begin{cases} h(\ell^-+\frac1n)+h'(\ell^-+\frac1n)(x-\ell^--\frac1n)+n|x-\ell^-|&\mbox{ for }x\leq\ell^-\\
				h(\ell^-+\frac1n)+h'(\ell^-+\frac1n)(x-\ell^--\frac1n)&\mbox{ for }x\in]\ell^-,\ell^-+\frac 1n]\\
					h(x)&\mbox{ for }x\geq \ell^-+\frac1n.\end{cases}$$
In this case as well, it is easy to check that all the required properties are satisfied. \end{proof}

{\bf Proof of Theorem \ref{mainthm1}.}

\begin{proof}
Let us start from the case where $g$ is $W^{1,1}$ and bounded from below, and $h$ is $C^2$, superlinear, with $h''>0$, and $\Om$ is a bounded convex set. A minimizer $\bar\ro$ exists (by the compactness of $\pical_2(\Om)$ and by the lower  semicontinuity of the functional  with respect to the weak convergence of measures). Thanks to Corollary \ref{coro:h and W},  there exists a Kantorovich potential $\varphi$ for the transport from $\bar\ro$ to $g$ such that $h'(\bar\ro)=\max\{C-\varphi,h'(0)\}$. This shows that $h'(\bar\ro)$ is Lipschitz continuous. Hence, $\bar\ro$ is bounded. On bounded sets $h'$ is a diffeomorphism with Lipschitz inverse, thanks to $h''>0$, which proves that $\bar\ro$ itself is Lipschitz. Then we can apply
Corollary \ref{cormain} and get
\begin{equation*}
\int_{\Omega} \Big(\nabla \bar\ro\cdot \frac{\nabla \varphi}{|\nabla\varphi|} +\nabla g\cdot \frac{\nabla \psi}{|\nabla \psi|}\Big)\,\dd x\ge 0.
\end{equation*}
Yet, a.e. on $\{\nabla\bar\ro   \neq 0\}$ we have from $h'(\bar\ro)=C-\varphi$. Using also $h''>0$, we get that $\nabla\varphi$ and $\nabla\bar\ro$ are vectors with opposite directions. Hence we have
\begin{equation*}
\int_{\Omega}| \nabla\bar\ro|\,\dd x\leq \int_{\Omega}\nabla g\cdot \frac{\nabla \psi}{|\nabla \psi|}\,\dd x\leq \int_{\Omega}|\nabla g|\,\dd x,
\end{equation*}
which is the desired estimate.

We can generalize to $h\in\mathcal C$ by using the previous lemma and approximating it with a sequence $h_n$. Thanks to monotone convergence we have $\Gamma-$convergence for the minimization problem that we consider. We also have compactness since $\pical_2(\Om)$ is compact, and uniqueness of the minimizer. Hence, the minimizers $\bar\ro_n$ corresponding to $h_n$ satisfy $\int_{\Omega}| \nabla\bar\ro_n|\leq \int_{\Omega}|\nabla g|$ and converge to the minimizer $\bar\ro$ corresponding to $h$. By the semicontinuity of the total variation we conclude the proof in this case.

Similarly, we can generalize to other convex functions $h$, approximating them with functions in $\mathcal C$ (notice that this is only interesting if the function $h$ allows the existence of at least a probability density with finite cost, i.e. if $h(1/|\Om|)<+\infty$). Also, we can take $g\in BV$ and approximate it with $W^{1,1}$ functions bounded from below. If the approximation is done for instance by convolution, then we have a sequence with $W_2(g_n,g)\to 0$, which guarantees uniform convergence of the functionals, and hence $\Gamma-$convergence.

We can also handle the case where $\Om$ is unbounded and convex, by first taking $g$ to be such that its support is a convex  bounded set, and $h\in \cC$. In this case the optimal $\bar\ro$ must be compactly supported as well. Indeed, the optimality condition $h'(\bar\ro)=\max\{C-\varphi,h'(0)\}$ imposes $\bar\ro=0$ on the set where $\varphi>C-h'(0)$. But on $\{\bar\ro>0\}$ we have $\varphi=\psi^c$, where $\psi$ is the Kantorovich potential defined on $\spt(g)$, which is bounded. Hence $\varphi$ grows at infinity quadratically, from $\varphi(x)=\inf_{y\in\spt(g)}\frac12|x-y|^2-\psi(y)$, which implies that there is no point $x$ with $\bar\ro(x)>0$ too far from $\spt (g)$. Once we know that the densities are compactly supported, the same arguments as above apply (note that being \(\Omega\) convex we ca assume that the densities are supported on a bounded convex set). Then one passes to the limit obtaining the result for any generic convex function $h$, and then we can also approximate $g$ (as above, we select a sequence $g_n$ of compactly supported densities converging to $g$ in $W_2$). Notice that in this case the convergence is no more uniform on $\pical_2(\Om)$, but it is uniform on a bounded set $W_2(\ro,g)\leq C$ which is the only one interesting in the minimization.
\end{proof}

\section{Projected measures under density constraints}\label{5}
\subsection{Existence, uniqueness,  characterization, stability of the projected measure}

In this section we will take $\Omega\subset\R^d$ be a given closed set with negligible boundary, $f:\Omega\to[0,+\infty[$ a measurable function in $L^1_{\rm loc}(\Omega)$ with $\int_\Omega f\,\dd x>1$ and $\mu \in \mathcal P_2(\Omega)$ a given probability measure on $\Omega$. 
We will consider the following projection problem 
\begin{equation}\label{op}
\min_{\tm\in K_f}\, W_2^2(\tm,\mu),
\end{equation}  
where we set $K_f=\{\ro\in L^1_+(\Omega)\ :\ \int_\Omega \tm\,\dd x=1,\,\ro\leq f\}$.

This section is devoted to the study of the above projection problem. We first want to summarize the main known results. Most of these results are only available in the case $f=1$. 

{\bf Existence.} The existence of a solution to Problem \eqref{op} is a consequence of the direct method of calculus of variations. Indeed, take a minimizing sequence $\ro^n$; it is tight thanks to the bound $W_2(\ro^n,\mu)\leq C$; it admits a weakly converging subsequence and the limit minimizes the functional $W_2(\cdot,\mu)$ because of its semicontinuity and of the fact that the inequality $\ro\leq f$ is preserved.
We note that from the existence point of view, the case $f\equiv 1$ and the general case do not show any significant difference.

{\bf Characterization.} The optimality conditions, derived in \cite{aude phd} exploiting the strategy developed in \cite{MauRouSan} (in the case $f=1$, but they are easy to adapt to the general case) state the following: if $\ro$ is a solution to the above problem and $\varphi$ is a Kantorovich potential  in the transport from $\ro$ to $\mu$, then there exists a threshold $\ell\in \R$ such that
$$\ro(x)=\begin{cases} f(x),&\mbox{ if }\varphi(x)< \ell,\\
					0,&\mbox{ if }\varphi(x)> \ell,\\
					\in[0,f(x)],&\mbox{ if }\varphi(x)= \ell.\end{cases}$$
In particular, this shows that $\nabla\varphi=0$ $\ro-$a.e. on $\{\ro<f\}$ and, since $T(x)=x-\nabla\varphi(x)$,  that the optimal transport $T$ from $\ro$ to $\mu$ is the identity on such set. If $\mu=g\dd x$ is absolutely continuous, then one can write the Monge-Amp\`ere equation 
\[
\det(DT(x))=\ro(x)/g(T(x))
\]
and deduce $\ro(x)=g(T(x))=g(x)$ a.e. on $\{\ro<f\}$. This suggests a sort of saturation result for the optimal $\ro$, i.e. $\ro(x)$ is either equal to $g(x)$ or to $f(x)$ (but one has to pay attention to the case $\ro=0$ and also to assume that $g$ is absolutely continuous). 
					
{\bf Uniqueness.} For absolutely continuous measures $\mu=g\,\dd x$ and generic $f$ the uniqueness of the projection follows by Lemma \ref{strict conv W_2}. In the specific case $f=1$ and $\Omega$ convex the uniqueness was proved in \cite{MauRouSan,aude phd} by a completely different method. In this case, as observed  by A. Figalli, one can use displacement convexity along generalized geodesics. This means that if $\ro^0$ and $\ro^1$ are two solutions, one can take for every $t\in[0,1]$ the convex combination $T^t=(1-t)T^0+tT^1$ of the optimal transport maps $T^i$ from $g$ to $\ro^i$
and the curve $t\mapsto\ro^t:=((1-t)T^0+tT^1)_\#\mu $ in $\mathcal{P}_2$, interpolating from $\ro^0$ to $\ro^1$. It can be proven that $\ro^t$ still satisfies $\ro^t\leq 1$ (but this can not be adapted to $f$, unless $f$ is concave) and that $t\mapsto W_2^2(\ro^t,g)<(1-t)W_2^2(\ro^0,g)+tW_2^2(\ro^1,g)$, which is a contradiction to the minimality. The assumption on $\mu$ can be relaxed but we need to ensure the existence of optimal transport maps: what we need to assume, is that $\mu$ gives no mass to ``small'' sets (i.e. $(d-1)-$dimensional); see \cite{gigli mcc} for the sharp assumptions and notions about this issue. Thanks to this uniqueness result, we can define a projection operator $P_{K_1}:\pical_2(\Omega)\cap L^1(\Omega)\to\pical_2(\Omega)\cap L^1(\Omega)$  through 
$$P_{K_1}[g]:=\argmin \{W_2^2(\ro,g)\ :\ \ro\in K_1\}.$$

{\bf Stability.} From the same displacement interpolation idea, A. Roudneff-Chupin also proved (\cite{aude phd}) that the projection is H\"older continuous with exponent $1/2$ for the $W_2$ distance whenever $\Omega$ is a compact convex set. We do not develop the proof here, we just refer to Proposition 2.3.4 of \cite{aude phd}. Notice that the constant in the H\"older continuity depends a priori on the diameter of $\Omega$. However, to be more precise, the following estimate is obtained (for $g^0$ and $g^1$ absolutely continuous)
\begin{equation}\label{holder}
W_2^2(P_{K_1}[g^0],P_{K_1}[g^1])\leq W_2^2(g^0,g^1)+W_2(g^0,g^1)(\dist(g^0,K_1)+\dist(g^1,K_1)),
\end{equation}
which shows that, even on unbounded domains, we have a local H\"older behavior.

In the rest of the section, we want to recover similar results in the largest possible generality, i.e. for general $f$, and without the assumptions on $\mu$ and $\Omega$.

We will first get a saturation characterization for the projections, which will allow for a general uniqueness result. Continuity will be an easy corollary.

In order to proceed, we first need the following lemma.

\begin{lemma}\label{rho_eq_f}
Let $\ro$ be a solution of the Problem \ref{op}. Let moreover $\g\in\Pi(\ro,\mu)$ be the optimal plan from $\ro$ to $\mu.$ If $(x_0,y_0)\in\spt(\g)$ then $\ro=f$ a.e. in $B(y_0,R),$ where $R=|y_0-x_0|.$ 
\end{lemma}
\begin{proof}
Let us suppose that this is not true and there exists a compact set  $K\subset B(y_0,R)$  with positive Lebesgue measure such that $\ro< f$ a.e. in $K$. Let  $\e:=\dist(\partial B(y_0,R),K)>0$.

By the definition of the support, for all $r>0$ we have that 
$$0<\g(B(x_0,r)\times B(y_0,r))\le \int_{B(x_0,r)} \ro\,\dd x \le \int_{B(x_0,r)} f\,\dd x.$$
By the absolute continuity of the integral, for $r>0$ small enough there exists $0<\a\le 1$ such that 
$$\g(B(x_0,r)\times B(y_0,r))=\a\int_K (f-\ro)\,\dd x=:\a m.$$
Now we construct the following measures $\tilde{\g},\eta\in\cP(\Om\times\Om)$ as
$$\tilde{\g}:=\g-\g\mres(B(x_0,r)\times B(y_0,r))+\eta \quad {\rm and}\quad \eta:=\a(f-\ro)\dd x\mres K\otimes (\pi^y)_\#\g\mres(B(x_0,r)\times B(y_0,r)).$$
It is immediate to check that $(\pi^y)_\#\tilde{\g}=\mu.$ On the other hand
$$\tilde{\ro}:=(\pi^x)_\#\tilde{\g}=\ro-\ro\mres B(x_0,r)+\a(f-\ro)\mres K\le f$$
is an admissible competitor in Problem \eqref{op} and we have the following
\begin{align*}
W_2^2(\tilde{\ro},\mu)&\le \int_{\Omega\times\Omega}|x-y|^2\,\dd\tilde{\g}(x,y)\\
& \le W_2^2(\ro,g)-\int_{B(x_0,r)\times B(y_0,r)}|x-y|^2\,\dd\g(x,y)+\int_{K\times B(y_0,r)}|x-y|^2\,\dd\eta(x,y)\\
&\le W_2^2(\ro,g)-(R-2r)^2\a m + (R-\e+r)^2\a m.
\end{align*}
Now if we chose $r>0$ small enough to have $R-2r>R-\e+r,$ i.e. $r<\e/3$ we get that 
$$W_2^2(\tilde{\ro},g)<W_2^2(\ro,g),$$
which is clearly a contradiction, hence the result follows.
\end{proof}

The following proposition establishes uniqueness of the projection on \(K_f\) as well as a very precise description of it. For a given measure \(\mu\) we are going to denote by \(\mu^{\rm ac}\) the density of its absolutely continuous part with respect to the Lebesgue measure, i.e.
\[
\mu=\mu^{\rm ac} \dd x+ \mu^{s},
\]
 with \(\mu^s\perp \dd x\). The following result recalls corresponding results in the {\it partial transport problem} (\cite{figa2}).
\begin{proposition}\label{existlemma}
Let $\Omega\subset\R^d$ be a convex set and let $f\in L_{\rm loc}^1(\Omega) $, \(f\ge 0\) be such that  $ \int_{\Omega}f\ge 1$. Then, for every probability measure $\mu \in \mathcal P(\Omega)$, there is a unique solution $\tm$ of the problem \eqref{op}. Moreover, $\tm$ is of the form 
\begin{equation}\label{structproj}
\tm=\mu^{\rm ac} \ind_{B}+f \ind_{B^c},
\end{equation}
for a measurable set $B\subset\Omega$.
\end{proposition}
\begin{proof} We first note that by setting \(f=0\) on \(\Omega^c\) we can assume that \(\Omega=\R^d\). Existence of a solution in Problem \ref{op} follows by the direct methods in the calculus of variations by noticing that the set \(K_f\) is closed with respect to the weak convergence of measures.

Let us prove now the saturation result \eqref{structproj}. Let us first premise the following fact:  if  $\mu,\nu\in\cP(\Omega)$,  $\g\in\Pi(\mu,\nu)$ and   we define the set 
$$
A(\gamma):=\{x\in\Om: {\rm the\ only\ point\ } (x,y)\in\spt (\g)\ {\rm is\ } (x,x)\},
$$
then 
\begin{equation}\label{treno1}
\mu\mres A(\gamma)\le \nu\mres A(\gamma). 
\end{equation}
In particular  \(\mu^{\rm ac}\le \nu^{\rm ac}\) for\ a.e.\ \(x\in A(\g)\). To prove \eqref{treno1}, let  $\phi\geq 0$ and write
\begin{equation*}
\begin{split}
\int_{A(\gamma)} \phi\,\dd\mu=\int \phi(x)\ind_{A(\gamma)}(x)\,\dd\gamma(x,y)&=\int \phi(x)\ind_{A(\gamma)}^2(x)\,\dd\gamma(x,y)\\
&=\int \phi(y)\ind_{A(\gamma)}(y)\ind_{A(\gamma)}(x)\,\dd\gamma(x,y)\\
&\leq \int \phi(y)\ind_{A(\gamma)}(y)\,\dd\gamma(x,y)=\int_{A(\gamma)}\phi\,\dd\nu,
\end{split}
\end{equation*}
where we used the fact that $\gamma-$a.e. $\ind_{A(\gamma)}(x)>0$ implies $x=y$.

Now, for  an optimal transport plan $\g\in\Pi(\ro,\mu)$, let us define 
\[
B:=\Leb(f)\cap\Leb(\mu^{\rm ac})\cap\Leb(\ro)\cap\{\ro<f\}^{(1)} \cap A(\gamma)^{(1)}\cap A(\tilde\gamma)^{(1)}.
\]
Here $\tilde\gamma\in\Pi(g,\ro)$ is the transport plan obtained by seeing $\gamma$ ``the other way around'', i.e. $\tilde\gamma$ is the image of $\gamma$ through the maps $(x,y)\mapsto (y,x)$ while  \(\Leb (h)\) is the set of  Lebesgue points of \(h\) and for a set \(A\) we denote by  \(A^{(1)}:=\Leb(\ind_A)\) the set of its  density  one points.

\medskip 
Let now  $x_0\in B$ and let us consider the following two cases:

\medskip
\noindent
\textit{Case 1.} $\ro(x_0)<\mu^{\rm ac} (x_0)$. Since, in particular, $\mu^{\rm ac} (x_0)>0$ and  \(x_0\in \Leb(\mu^{\rm ac})\) we have that  $x_0\in\spt(\mu ).$ From Lemma \ref{rho_eq_f} wee see that $(y_0,x_0)\in\spt(\g)$ implies $y_0=x_0$. Indeed if this were not the case   there would exist a ball where $\ro=f$ a.e. and $x_0$ would be in the middle of this ball; from $x_0\in \Leb(f)\cap\Leb(\ro)$ we would get $\ro(x_0)=f(x_0)$ a contradiction with \(x_0\in B\).  Hence, if we use the set $A(\tilde\g)$ defined above with   $\nu=\ro$, we have $x_0\in A(\tilde\g)$. From $x_0\in \Leb(\mu^{\rm ac})\cap\Leb(\ro)$ we get $\mu^{\rm{ac}} (x_0)\le \ro(x_0),$ which is a contradiction.

\textit{Case 2.} $\mu^{\rm ac} (x_0)<\ro(x_0).$ Exactly as in the previous case we have that $x_0\in\spt(\ro)$ and, by the Lemma \ref{rho_eq_f}, we have again that $(x_0,y_0)\in\spt(\g)$ implies $y_0=x_0$. Indeed, otherwise  $x_0$ would be on the boundary of a ball where $\ro=f$ a  contradiction with \(x_0\in \{\ro<f\}^{(1)}\). Hence, we get $x_0\in A(\g)$ and $\ro(x_0)\leq \mu^{\rm ac} (x_0)$,   again a contradiction.

Hence we get that \(\mu^{\rm{ac}}=\ro\) for \(x\in B\). By the definition of \(B\),
\[
B^c\subset_{\rm  a.e.} \{\ro=f\}\cup A(\gamma)^c\cup A(\tilde \gamma)^c\,,
\]
where a.e. refers to the Lebesgue measure. By applying Lemma \ref{rho_eq_f},  this implies that \(\varrho=f\)   a.e. on \(B^c\), and  concludes the proof of \eqref{structproj}.

Uniqueness of the projection it is now an immediate consequence of the  saturation property \eqref{structproj}. Indeed, suppose that $\ro_0$ and $\ro_1$ were two different projections of a same measure $g$. Define $\ro_{1/2}=\frac 12 \ro_0+\frac 12 \ro_1$. Then, by convexity of $W_2^2(\cdot,\mu)$, we get that $\ro_{1/2}$ is also optimal. But its density is not saturated on the set where the densities of $\ro_0$ and $\ro_1$ differ, in  contradiction with \eqref{structproj}.
\end{proof}
%
%

\begin{corollary} \label{cor:cont}For fixed $f$, the map $P_{K_f}:\pical_2(\Omega)\to\pical_2(\Omega)$ defined through 
$$P_{K_f}[\mu ]:=\argmin \{W_2^2(\ro,\mu)\ :\ \ro\in K_f\}$$ is continuous in the following sense: if $\mu_n\to \mu$ for the $W_2$ distance, then $P_{K_f}[\mu_n]\deb P_{K_f}[\mu]$ in the weak convergence.

Moreover, in the case where $f=1$ and $\Omega$ is a convex set, the projection is also locally $\frac 12-$H\"older continuous for $W_2$ on the whole $\pical_2(\Omega)$ and satisfies \eqref{holder}.
\end{corollary}
\begin{proof}
This is just a matter of compactness and uniqueness. Indeed, take a sequence $\mu_n\to \mu$ and look at $P_{K_f}[\mu_n]$. It is a tight sequence of measures since 
\begin{equation}\label{treno}
W_2(P_{K_f}[\mu_n],\mu)\leq W_2(P_{K_f}[\mu_n],\mu_n)+W_2(\mu_n,\mu )\leq W_2(\ro,\mu)+2W_2(\mu_n,\mu)\,,
\end{equation}
where $\ro\in K_f$ is any admissible measure.  Hence we can extract a weakly converging subsequence to some measure \(\tilde \ro \in K_f\) (recall that \(K_f\) is weakly closed). Moreover, by the lower semicontinuity of $W_2$ with respect to  the weak convergence and since \(W_2(\mu_n,\mu)\to0\), passing to the limit in \eqref{treno} we get
\[
W_2(\tilde \ro, \mu)\le W_2(\ro,\mu)\qquad \forall\, \ro \in K_f.
\]
Uniqueness of the projection implies \(\tilde \ro=P_{K_f} (\mu)\) and thus that the limit is independent on the extracted subsequence, this proves the desired continuity.

Concerning the   second part of the statement, we take arbitrary $\mu^1$ and $\mu ^2$ (not necessarily absolutely continuous) and we approximate them in the $W_2$ distance with absolutely continuous measures $g^i_n$ ($i=1,2$; for instance by convolution), then we have, from \eqref{holder} 
$$W_2^2(P_{K_1}[g^0_n],P_{K_1}[g^1_n])\leq W_2^2(g^0_n,g^1_n)+W_2(g^0_n,g^1_n)(\dist(g^0_n,K_1)+\dist(g^1_n,K_1)),$$
 and we can pass to the limit as $n\to \infty$.
\end{proof}

The following technical lemma will be used in the next section and establishes the continuity of the projection with respect to \(f\). To state it let us consider, for given \(f\in L^1_{\rm loc}\) and \(\mu \in \mathcal P_2(\Omega)\), the following functional 
\[
\mathcal F_f(\tm):=
\begin{cases}
\frac12 W_2^2(\mu,\tm),\quad&\textrm{if  $\ro \in K_f $}\\
+\infty, &\textrm{otherwise.}
\end{cases}
\]
Proposition \ref{existlemma} can be restated by saying that the functional $\mathcal F_f$ has a unique minimizer in $\mathcal{P}_2(\Om).$

\begin{lemma}\label{gammaconvfnf}

Let \(f_n\,, f\in L^1_{\rm loc}(\Omega)\) with \(\int_{\Omega} f_n\,\dd x\ge 1,\int_{\Omega} f\,\dd x\geq 1\) and let us assume that \(f_n\to f\) in \(L^1_{\rm loc}(\Omega)\) and almost everywhere. Also assume $f_n\in \mathcal P_2(\Omega)$ if $\int_\Omega f_n\,\dd x=1$ and $f\in \mathcal P_2(\Omega)$ if $\int_\Omega f\,\dd x=1$. Then, for every $\mu\in \mathcal P_2(\Omega)$,
\begin{enumerate}
\item[(i)] The sequence $(P_{K_{f_n}}(\mu))_n$ is tight.
\item[(ii)] We have \(P_{K_{f_n}}(\mu)\deb P_{K_f}(\mu)\).
\item[(iii)] If $\int_\Omega  f>1,$ then \(\mathcal F_{f_n}\) \(\Gamma-\)converges to \(\mathcal F_{f}\) with respect to the weak convergence of measures. 
\end{enumerate}
\end{lemma}

\begin{proof} 
Let us denote by $\bar\varrho_n$ the projection $P_{K_{f_n}}(\mu)$ and let us start from proving its tightness, i.e. {\it (i)}. We fix $\ve>0$: there exists a radius $R_0$ such that $\mu(B(0,R_0))>1-\frac\ve 2$ and $\int_{B(0,R_0)}f>1-\frac\ve 2$. By $L^1_{\mathrm{loc}}$ convergence, there exists $n_0$ such that $\int_{B(0,R_0)}f_n>1-\ve$ pour $n>n_0$. Now, take $R>3R_0$ and suppose $\bar\varrho_n(B(0,R)^c)>\ve$ for $n\geq n_0$. Then, the optimal transport $T$ from $\bar\varrho_n$ to $\mu$ should move some mass from $B(0,R)^c$ to $B(0,R_0)$. Let us take a point $x_0\in B(0,R)^c$ such that $T(x_0)\in B(0,R_0)$. From Lemma \ref{rho_eq_f}, this means that $\bar\varrho_n=f_n$ on the ball $B(T(x_0),|x_0-T(x_0)|)\supset B(T(x_0),2R_0)\supset B(0,R_0)$. But this means $\int_{B(0,R_0)}\bar\varrho_n=\int_{B(0,R_0)}f_n>1-\ve$, and hence $\bar\varrho_n(B(0,R)^c)\leq\ve,$ which is a contradiction. This shows that $\bar\varrho_n$ is tight.

Now, if $\int_\Omega f=1$, then the weak limit of $\bar\varrho_n$ (up to subsequences) can only be $f$ itself, since it must be a probability density bounded from above by $f$ and $f=P_{K_f}(\mu)$. This proves {\it (ii)} in the case $\int_\Omega f=1$. In the case $\int_\Omega f>1$, this will be a consequence of {\it (iii)}. Notice that in this case we necessarily have $\int_\Omega f_n>1$ for $n$ large enough.

Let us prove {\it (iii)}.
Since \(\varrho_n\le f_n\), \(\varrho_n\deb \varrho\) and \(f_n\to f\) in \(L^1_{\rm loc}\) immediately implies that \(\varrho\le f\), the \(\Gamma-\)liminf inequality simply follows by the lower semicontinuity of \(W_2\).

Concerning the $\Gamma-$limsup, we need to prove that every density $\ro\in \pical_2(\Omega)$  with $\ro\leq f$ a.e. can be approximated by a sequence $\ro_n\leq f_n$ a.e. with $W_2(\ro_n,g)\to W_2(\ro,g)$. In order to do this let us  define $\tilde\ro_n:=\min\{\ro,f_n\}$. Note that  $\tilde\ro_n$ is  not admissible since it is not a probability, because in general $\int\tilde\ro_n<1$. Yet, we have $\int\tilde\ro_n\to  1$ since $\tilde\ro_n\to\min\{\ro,f\}=\ro$ and this convergence is dominated by $\ro$.  We want to ``complete'' $\tilde\ro_n$ so as to get a probability, stay admissible, and converge to $\ro$ in $W_2$, since this will imply that $W_2(\ro_n,\mu)\to W_2(\ro,\mu)$. 

Let us select a ball $B$ such that $\int_{B\cap \Omega} f>1$ and note that we can find $\ve>0$ such that the set $\{f>\ro+\ve\}\cap B$ is of positive measure, i.e. $m:=|\{f>\ro+\ve\}\cap B|>0$. Since $f_n\to f$ a.e., the set $B_n:=\{f_n>\ro+\frac\ve 2\}\cap B$ has measure larger than $m/2$ for large $n$. Now take $B_n'\subset B_n$ with $|B_n'|=\frac 2 \ve(1-\int \tilde\ro_n)\to 0$, and define 
\[
\ro_n:=\tilde\ro_n+\frac\ve 2 \ind_{B_n'}.
\]
By construction, $\int\ro_n=1$ and $\ro_n\leq f_n$ a.e. since on $B_n'$ we have $\tilde\ro_n=\ro$ and $\ro+\frac\ve 2<f_n$ while on  the complement of $B_n'$, $\tilde\ro_n\leq f_n$ a.e. by definition. To conclude the proof we   only need to check $W_2(\ro_n,\ro)\to 0$. This is equivalent (see, for instance, \cite{AmbGigSav} or \cite{villani}) to 
\begin{equation}
\int \phi\ro_n\to\int\phi \ro
\end{equation}
 for all continuous functions $\phi$ with such that \(\phi\le C(1+|x|^2)\). Since \(\ro\in \mathcal P_2(\Omega)\) and \(\tilde \ro_n\le\ro\), thank to the dominated convergence theorem it is enough to show that \(\int \phi(\ro_n-\tilde \ro_n)\to 0\). But \( \ro_n-\tilde \ro_n\) converges to \(0\) in \(L^1\) and it is supported in \(B_n' \subset B\). Since \(\phi\) is bounded on \(B\) we obtain the desired  conclusion.
\end{proof}

\begin{remark}
{\rm 
 Let us conclude this section with the following {\it open question}: for $f=1$ the projection is continuous and we can even provide H\"older bounds on $P_{K_1}$. 
The question whether $P_{K_1}$ is $1$-Lipschitz, as far as we know, is open. Let us underline that some sort of $1$-Lipschitz results have been proven in \cite{CarCra} for solutions of similar variational problems, but seem impossible to adapt in this framework.

For the case $f\neq1$ even the continuity of the projection with respect to the Wasserstein distance seems delicate.
 
%
 }
 \end{remark}

\subsection{BV estimates for $P_{K_f}$} In this section, we  prove Theorem \ref{mainthm2}. Notice that the case $f=1$ has already been proven as a particular case of Theorem \ref{mainthm1}. To handle the general case, we develop a slightly different strategy,  based on  the standard idea to approximate $L^\infty$ bounds with $L^p$ penalizations.

Let  $m\in\N$ and let us assume that  $\inf f>0$, for \(\mu\in \mathcal P_2(\Omega)\), we define the approximating functionals $\mathcal{F}_m:L_{+}^1(\Omega)\to\R\cup\{+\infty\}$ by
\begin{equation*}
\mathcal{F}_{m}(\tm):=\frac12 W_2^2(\mu,\tm)+\frac{1}{m+1}\int_\Omega \left(\frac{ \tm }{ f}\right)^{m+1}\,\dd x+ \frac{\ve_m}{2}\int_\Omega \left(\frac{ \tm }{ f}\right)^{2}\,\dd x
\end{equation*}
and the  limit functional $\mathcal F$ as 
\[
\mathcal F(\tm):=
\begin{cases}
\frac12 W_2^2(\mu,\tm),\quad&\textrm{if  $\ro \in K_f $}\\
+\infty, &\textrm{otherwise}
\end{cases}
\]
Here $\ve_m\downarrow 0$ is a small  parameter to be chosen later.
\bigskip
%
%

\begin{lemma}\label{approxlemma}
Let $\Omega\subset\R^d$ and $f:\Omega\to (0,+\infty)$ be a measurable function, bounded from below and from above by positive constants  and let \(\mu\in \mathcal{P}_2(\Omega)\). Then:
\begin{enumerate}
\item[(i)] There are unique minimizers $\tm$, $\tm_{m}$ in $L^1(\Omega)$ for each of the functionals $\mathcal{F}$ and $\mathcal{F}_{m}$, respectively.
\item[(ii)] The family of functionals $\mathcal{F}_m$ $\Gamma$-converges for the weak convergence of probability measures to $\mathcal{F}$, and the minimizers $\tm_m$ weakly converge to $\tm$, as $m\to \infty$.
\item[(iii)] The minimizers $\tm_m$ of $\mathcal{F}_m$ satisfy 
\begin{equation}\label{optconepsm}
\varphi_{m}+\left(\frac{\tm_{m}}{f}\right)^m\frac{1}{f} +\ve_m\left(\frac{\tm_{m}}{f}\right)\frac{1}{f}=0,\end{equation}
for a suitable Kantorovich potential $\varphi_m$ in the transport from $\tm_m$ to $\mu$.
\end{enumerate}   
\end{lemma} 
\begin{proof} Existence and uniqueness  of minimizers of \(\mathcal F\) has been established in Proposition \ref{existlemma}. Existence of minimizers of \(\mathcal F_m\) is again a simple application of the direct methods in the calculus of variations and uniqueness follows from strict convexity.

%

Let us prove the $\Gamma-$convergence in {\it (ii)}. In order to prove the $\Gamma-$liminf inequality, let  $\tm_m\deb\tm$. If  $\mathcal F_m(\tm_m)\leq C$, then for every $m_0\leq m$ and every finite measure set $A\subset\Omega$, we have 
\[
\|\tm_m/f\|_{L^{m_0}(A)}\leq |A|^{\frac{1}{m_0}-\frac{1}{m+1}} (C(m+1))^\frac{1}{m+1}.
\]
 If we pass to the limit $m\to\infty$, from $\frac{\tm_m}{f}\deb  \frac{\tm}{f}$, we get $||\tm/f||_{L^{m_0}(A)}\leq |A|^\frac{1}{m_0}$. Letting \(m_0\) go to infinity we obtain  $||\tm/f||_{L^{\infty}}\leq 1$, i.e. \(\ro \in K_f\). Since
 \[
 \mathcal F_m(\ro_m)\ge \frac 1 2 W_2^2(\mu,\ro_m),
 \]
 the lower semicontinuity of \(W_2^2\) with respect to weak converges proves the \(\Gamma-\)liminf inequality.
 
In order to prove  $\Gamma-$limsup, we use the constant sequence  $\tm_m=\tm$ as a recovery sequence. Since we can assume $\tm\leq f$ (otherwise there is nothing to prove, since  $\mathcal F(\tm)=+\infty$), it is clear that the second and third parts of the functional tend to $0$, thus proving the desired inequality.

The last part of the statement finally follows from Theorem \ref{struct} (vi) and Lemma \ref{lemma:h}, exactly as in Corollary \ref{coro:h and W}.
\end{proof}


{\bf Proof of Theorem \ref{mainthm2}}


\begin{proof} Clearly we can assume that $TV(g,\Omega) $ and   $TV(f,\Omega)$ are finite and that \(\int_\Omega f >1\) since otherwise the conclusion is trivial. 

\medskip
\noindent
\textit{Step 1.} Assume that the support of $g$ is compact, that $f\in C^\infty(\Omega)$  is bounded from above and below by positive constants,  and let  $\tm_{m}$ be the  minimizer of  $\mathcal{F}_{m}$. As in the proof of Theorem \ref{mainthm1}, we can use the optimality condition \eqref{optconepsm} to prove that $\tm$ is compactly supported. Also, the same condition imply that $\tm$ is Lipschitz continuous. Indeed, we can write \eqref{optconepsm} as
$$\varphi f + H_m'\left(\frac{\tm}{f}\right)=0,$$
where $H_m(t)=\frac{1}{m+1}t^{m+1}+\frac{\ve_m}{2}t^{2}.$  Since $H_m$ is smooth and convex and $H_m''$ is bounded from below by a positive constant $H_m'$ is  invertible and   
$$\ro=f \cdot(H_m')^{-1}(-\varphi f),$$
where $ (H_m')^{-1}$ is Lipschitz continuous. Since $\varphi$ and $f$ are locally Lipschitz, this gives Lipschitz continuity for $\ro$ on a neighborhood of its support.

Taking the derivative  of the optimality condition \eqref{optconepsm} we obtain
$$\nabla\varphi_{m} + \left(m\left(\frac{\tm_{m}}{f}\right)^{m-1}+\ve_m\right)\frac{f\nabla \tm_{m}-\tm_{m}\nabla f}{f^3}-\left(\left(\frac{\tm_{m}}{f}\right)^m+\ve_m\frac{\tm_m}{f}\right)\frac{\nabla f}{f^2}=0.$$
Rearranging the terms we have
$$\nabla\varphi_{m} + A\nabla\tm_{m}- B \nabla f=0,$$
where by $A$ and $B$ we denote the (positive!) functions 
$$A:=\left(m\left(\frac{\tm_{m}}{f}\right)^{m-1}\!\!+\!\ve_m\right)\frac{1}{f^2}\quad \hbox{and} \quad B:=\left(m\left(\frac{\tm_{m}}{f}\right)^{m-1}\!\!+\!\ve_m\right)\frac{\tm_m}{f^3}+\left(\left(\frac{\tm_{m}}{f}\right)^m+\ve_m\frac{\tm_m}{f}\right)\frac{1}{f^2}.$$ 
Now we will use the inequality from Corollary \ref{cormain} for $\tm_{m}$ and $g$ in the form
$$\int_{\Omega}|\nabla \tm_{m}|\,\dd x\le\int_{\Omega}|\nabla g|\,\dd x+\int_{\Omega}\nabla \tm_{m}\cdot\left(\frac{\nabla \tm_{m}}{|\nabla \tm_{m}|}+\frac{\nabla\varphi_{m}}{|\nabla\varphi_{m}|}\right)\,\dd x.$$
In order to estimate the second integral on the right-hand side we use the inequality 
\begin{equation}\label{ineqab}
\left|\frac{a}{|a|}-\frac{b}{|b|}\right|\le \left|\frac{a}{|a|}-\frac{b}{|a|}\right|+\left|\frac{b}{|a|}-\frac{b}{|b|}\right|= \frac{|a-b|}{|a|}+\frac{|b|-|a|}{|a|}\le \frac{2}{|a|}|a-b|,
\end{equation}
for all non-zero $a,b\in\R^d$ (that we apply to $a=A\nabla\ro_m$ and $b=-\nabla\varphi_m$), and we obtain 
\begin{align*}
\int_{\Omega}|\nabla \tm_{m}|\,\dd x & \le \int_{\Omega}|\nabla \gm|\,\dd x + \int_{\Omega}|\nabla \tm_{m}|\cdot\left|\frac{A\nabla \tm_{m}}{A|\nabla \tm_{m}|}+\frac{\nabla\varphi_{m}}{|\nabla\varphi_{m}|}\right|\,\dd x\\
&\le \int_{\Omega}|\nabla \gm|\,\dd x + 2 \int_{\Omega}\frac{1}{A}\big|A\nabla \tm_{m}+\nabla\varphi_{m}\big|\,\dd x\\
&\le \int_{\Omega}|\nabla \gm|\,\dd x + 2 \int_{\Omega}\frac{B}{A}|\nabla f|\,\dd x.
\end{align*}
We must now estimate the ratio $B/A$. If we denote by $\lambda$ the ratio $\tm_m/f$ we may write
$$\frac B A = \lambda+\lambda\frac{\ve_m+\lambda^{m-1}}{\ve_m+m\lambda^{m-1}}\leq \lambda\left(1+\frac 1m\right)+\frac{\ve_m\lambda}{\ve_m+m\lambda^{m-1}}.$$
Now, consider that 
$$\max_{\lambda\in\R_+}\frac{\ve_m\lambda}{\ve_m+m\lambda^{m-1}}=\frac{m-2}{m-1}\left(\frac{\eps_m}{m(m-2)}\right)^{1/(m-1)}=:\delta_m$$
 is a quantity depending on $m$ and tending to $0$ if $\ve_m$ is chosen small enough (for instance $\ve_m=2^{-m^2}$). This allows to write
 $$\int_{\Omega}|\nabla \tm_{m}|\,\dd x\leq \int_{\Omega}|\nabla \gm|\,\dd x + 2\left(1+\frac 1m\right) \int_{\Omega}\frac{\tm_m}{f}|\nabla f|\,\dd x+2\delta_m\int_{\Omega}|\nabla f|\,\dd x.$$
 
 In the limit, as $m\to+\infty$, we obtain 
\begin{align*}
\int_{\Omega}|\nabla \tm|\,\dd x&\le \int_{\Omega} |\nabla \gm|\,\dd x+2 \int_{\Omega} \frac{\tm}{f}|\nabla f|\,\dd x.
\end{align*}
Using the fact that $\ro\le f$, we get 
\begin{align*}
\int_{\Omega}|\nabla \tm|\,\dd x\le\int_{\Omega}|\nabla \gm|\,\dd x+2 \int_{\Omega}|\nabla f|\,\dd x.
\end{align*}
%
%

\medskip
\noindent 
\textit{Step 2.} To treat the case  $\gm, f\in BV_{\rm loc} (\Omega)$ we  proceed  by approximation as  in the proof of  Theorem  \ref{mainthm1}. To do this we just note that Corollary \ref{cor:cont} and  Lemma \ref{gammaconvfnf} give the desired continuity property of the projection with respect both to \(g\) and \(f\), lower semicontinuity of the total variation with respect to the weak convergence then implies the conclusion.
\end{proof}

\begin{remark}
{\rm
We conclude this section by underlining that the constant \(2\) in inequality \eqref{milano} can not be  replaced by any  smaller constant. Indeed if  $\Om=\R$, $f=\ind_{\R_+}$, $g=\frac 1n \ind_{[-n,0]}$ then  $\ro=P_{K_f}(g)=\ind_{[0,1]}$ and   $\int |\nabla \ro|=2$, $\int |\nabla f|=1$,  $\int |\nabla g|=\frac 2n$.
}
\end{remark}

\section{Applications}\label{6}
In this section we discuss some applications of Theorems \ref{mainthm1} and \ref{mainthm2} and we present some open problems.

\subsection{Partial transport}
The projection problem on \(K_f\) is a particular case of the so called \emph{partial transport problem},  see \cite{figa1,figa2}. Indeed, the problem is to transport $\mu$ to a part of the measure $f$, which is a measure with mass larger than $1$. As typical in the partial transport problem, the solution has an active region, which is given by $f$ restricted to a certain set. This set satisfies a sort of interior ball condition, with a radius depending on the distance between each point and its image. In the partial transport case some regularity ($C^{1,\alpha}$) is known  for the optimal map away  from the intersection of the supports of the two measures.

A natural question is how to apply the technique that we developed here in the framework of more general partial transport problems (in general, both measures could have mass larger than $1$ and could be transported only partially), and/or whether results or ideas from partial transport could be translated into the regularity of the free boundary in the projection.

\subsection{Shape optimization}

If we take a set $A\subset\R^d$ with $|A|<1$ and finite second moment $\int_A |x|^2\,\dd x<+\infty$, a natural question is which is the set $B$ with volume $1$ such that the uniform probability density on $B$ is closest to that on $A$. This means solving a shape optimization problem of the form
$$\min\{W_2^2(\ind_B,\frac{1}{|A|}\ind_A)\;:\; |B|=1\}.$$

The considerations in Section 5.1 show that solving such a problem is equivalent to solving
$$\min\{W_2^2(\ro,\frac{1}{|A|}\ind_A)\;:\; \ro\in\pical_2(\R^d)\}$$
and that the optimal $\ro$ is of the form $\ro=\ind_B$, $B\supset A$. Also, from our Theorem \ref{mainthm2} (with $f=1$), we deduce that if $A$ is of finite perimeter, then the same is true for $B$, and $\mathrm{Per}(B)\leq\frac{1}{|A|}  \mathrm{Per}(A)$ (i.e. the perimeter is bounded by the Cheeger ratio of $A$).

It is interesting to compare this problem with this perimeter bound with the problem studied in \cite{manolis}, which has the same words but in different order: more precisely: here we minimize the Wasserstein distance and we try to get an information on the perimeter, in  \cite{manolis} the functional to be minimized is a combination of $W_2$ and the perimeter. Hence, the techniques to prove any kind of results are different, because here $W_2$ cannot be considered as a lower order perturbation of the perimeter.

As a consequence, many natural questions arise: if $A$ is a nice closed set, can we say that $B$ contains $A$ in its interior? if $A$ is convex is $B$ convex? what about the regularity of $\partial B$?

\subsection{Set evolution problems}
Consider the following problem. For a given set $A\subset\R^d$ we define $\ro_0=\ind_A$. For a time interval $[0,T]$ and a time step $\t>0$ (and $N+1:=\left[\frac{T}{\tau}\right]$) we consider the following scheme $\ro_0^\tau:=\ro_0$ and
\begin{equation}\label{scheme1}
\ro_{k+1}^\tau:=P_{K_1}\left[(1+\t)\ro_k^\tau\right],\ k\in\{0,\dots,N-1\},
\end{equation}
(here we extend the notion of Wasserstein distance and projection to measures with the same mass, even if different from $1$: in particular, the mass of $\ro_k^\tau$ will be $|A|(1+\tau)^k$ and at every step we project $\ro_k^\tau$ on the set of finite positive measure, with the same mass of $\ro_k^\tau$, and with density bounded by $1$, and we still denote this set by $K_1$ and the projection operator in the sense of the quadratic Wasserstein distance onto this set by $P_{K_1}$).
We want to study the convergence of this algorithm as $\tau\to 0.$ This is a very simplified model for the growth of a biological population, which increases exponentially in size (supposing that there is enough food: see \cite{MauRouSan bacteries} for a more sophisticated model) but is subject to a density constraint because each individual needs a certain amount of space. Notice that this scheme formally follows the same evolution as in the Hele-Shaw flow (this can be justified by the fact that, close to uniform density  the $W_2$ distance and the $H^{-1}$ distance are asymptotically the same). 

Independently of the compactness arguments that we need to prove the convergence of the scheme, we notice that, for fixed $\tau>0$, all the densities $\ro_k^\tau$ are indeed indicator functions (this comes from the consideration in Section 5.1). Thus we have an evolution of sets. A natural question is whether this stays true when we pass to the limit as $\tau\to 0$. Indeed, we generally prove convergence of the scheme in the weak sense of measures, and it is well-known that, in general, a weak limit of indicator functions is not necessarily an indicator itself. However  Theorem \ref{mainthm2} provides an a priori bound  the perimeter of these sets. This \(BV\) bound allows to transform weak convergence as measures into strong $L^1$ convergence, and to preserve the fact that these densities are indicator functions.

Notice on the other hand that the same result could not be applied in the case where the projection was performed onto $K_f$, for a non-constant $f$. The reason lies in the term $2\int |\nabla f|$ in the estimate we provided. This means that, a priori, instead of being decreasing, the total variation could increase at each step of a fixed amount $2\int |\nabla f|$. When $\tau\to 0$, the number of iterations diverges and this does not allow to prove any \(BV\) estimate on the solution. Yet, a natural question would be to prove that the set evolution is well-defined as well, using maybe the fact that these sets are increasing in time.
 
 \subsection{Crowd movement with diffusion}
 
In \cite{MauRouSan, aude phd} crowd movement models where a density $\ro$ evolves according to a given vector field $v$, but subject to a density constraint $\ro\leq 1$ are studied. This means that, without the density constraint, the equation would be $\partial_t\ro+\nabla\cdot(\ro v)=0$, and a natural way to discretize the constrained equation would be to set $\tilde\ro^\tau_{k+1}=(id+\tau v)_\#\ro^\tau_{k}$ and then $\ro^\tau_{k+1}=P_{K_1}[\tilde\ro^\tau_{k+1}]$.

What happens if we want to add some diffusion, i.e. if the continuity equation is replaced by a Fokker-Planck equation $\partial_t\ro-\Delta\ro+\nabla\cdot(\ro v)=0$? among other possible methods, one discretization idea is the following: define $\tilde\ro^\tau_{k+1}$ by following the unconstrained Fokker-Planck equation for time $\tau$ starting from $\ro^\tau_{k}$, and then project. 
In order to get some compactness of the discrete curves we need to estimate the distance between  $\ro^\tau_{k}$ and $\tilde\ro^\tau_{k+1}$. It is not difficult to see that the speed of the solution of the Heat Equation (and also of the Fokker-Planck equation) for the distance $W_p$ is related to $\|\nabla \ro\|_{L^p}$. It is well known that these parabolic equations regularize and so the $L^p$ norm of the gradient will not blow up in time, but we have to keep into account the projections that we perform every time step $\tau$. From the discontinuities that appear in the projected measures, one cannot expected that $W^{1,p}$ bounds on $\ro$ are preserved. The only reasonable bound is for $p=1$, i.e. a \(BV\) bound, which is exactly what is provided in this paper. 

The application to crowd motion with diffusion has been studied by the second and third author in \cite{mesz_santamb}.

\subsection{BV estimates for some degenerate diffusion equation}

In this subsection we apply our main Theorem \ref{mainthm1} to establish \(BV\) estimates for for some degenerate diffusion equation.  \(BV\) estimates for these equations are usually  known   and they can be derived by looking at the evolution in time of the \(BV\) norm of the solution. Theorem \ref{mainthm1} allows to give an optimal transport proof of these estimates.
Let \(h:\R^+\to \R\) be a given super-linear convex function and let us consider the problem 
\begin{equation}
\label{pm}
\begin{cases}
\partial_t\ro_t=\nabla \cdot \left(h''(\ro_t)\rho_t\nabla \rho_t\right), & \mbox{\rm in}\ (0,T]\times\R^d,\\
\ro(0,\cdot)=\ro_0, & \mbox{\rm in}\ \R^d,
\end{cases}
\end{equation}
where $\ro_0$ is a non-negative \(BV\) probability density. We remark that by the evolution for any $t\in(0,T]$ $\ro_t$ will remain a non-negative probability density. In the case \(h(\rho)=\rho^{m}/(m-1)\) in equation \eqref{pm} we get precisely the {\em porous medium equation} \(\partial_t \rho=\Delta(\rho^m)\)  (see \cite{vazquez}). 

Since the seminal work of F. Otto (\cite{otto}) we know that the problem \eqref{pm} can be seen as a gradient flow of the functional 
\[
\ds\cF(\ro):=\int_{\R^d}h(\ro)
\]
 in the space $(\cP(\R^d),W_2).$ 
%
As a gradient flow, this equation can be discretized in time through an implicit Euler scheme. More precisely let us take   
a time step $\tau>0$  and let us consider the following scheme: $\ro_0^\tau:=\ro_0$ and
\begin{equation}\label{scheme2}
\ro_{k+1}^\tau:=\argmin_\ro\left\{\frac{1}{2\tau}W_2^2(\ro,\ro_k^\tau)+\int h(\ro)\right\},\  k\in\{0,\dots,N-1\}.
\end{equation}
 where $N:=\left[\frac{T}{\tau}\right]$.
Using piecewise constant and geodesic interpolations between the $\ro_k^\tau$'s with the corresponding velocities and momentums, it is possible to show that as $\tau\to 0$ we will get a curve $\ro_t,\ t\in[0,T]$ in $(\cP(\R^d),W_2)$ which solves 
 \[
 \begin{cases}
 \partial_t\rho_t+\nabla\cdot(\ro_t v_t)=0\\
 v_t =-h''(\ro_t)\nabla\ro_t,
 \end{cases}
 \]
 hence
$$\partial_t\ro_t-\nabla\cdot(h''(\ro_t) \ro_t \nabla\ro_t)=0,$$
that is \(\ro_t\) is a solution to \eqref{pm}, see \cite{AmbGigSav} for a rigorous presentation of these facts.

We now note that  Theorem \ref{mainthm1} implies that 
$$\int_{\R^d}|\nabla\ro_{k+1}^\tau|\,\dd x\le\int_{\R^d}|\nabla\ro_k^\tau|\,\dd x,$$
hence the total variation decreases for the sequence $\ro_0^\tau,\dots,\ro_N^\tau.$ As the estimations do not depend on $\tau>0$ this will remain true also in the limit $\tau\to 0.$ Hence (assuming uniqueness for the limiting  equation) we get that for any $t,s\in[0,T],\ t>s$
$$TV(\ro_t,\R^d)\le TV(\ro_s,\R^d),$$
and in particular for any $t\in[0,T]$
$$TV(\ro_t,\R^d)\le TV(\ro_0,\R^d).$$


\end{document}